\documentclass{siamonline250211}

\usepackage{amsmath}
\usepackage{amssymb}
\usepackage{mathtools}
\usepackage{hyperref}
\usepackage{wrapfig}
\usepackage{appendix}
\usepackage{algorithmic}
\Crefname{ALC@unique}{Line}{Lines}

\usepackage{tikz}
\usetikzlibrary{matrix,decorations.pathreplacing}

\usepackage{cutwin}

\usepackage{enumitem}
\setlist[enumerate]{leftmargin=.5in}
\setlist[itemize]{leftmargin=.5in}

\newcommand{\term}[1]{\textbf{#1}}
\newcommand{\tech}[1]{\textsf{#1}}
\newcommand{\op}[1]{\ensuremath{\operatorname{#1}}}

\newcommand{\zmat}[2]{\ensuremath{M_{#1 \times #2}(\Z)}}

\newcommand{\bolda}{\ensuremath{\mathbf{a}}}
\newcommand{\boldb}{\ensuremath{\mathbf{b}}}
\newcommand{\boldc}{\ensuremath{\mathbf{c}}}

\newcommand{\boldp}{\ensuremath{\mathbf{p}}}

\newcommand{\boldt}{\ensuremath{\mathbf{t}}}
\newcommand{\boldv}{\ensuremath{\mathbf{v}}}
\newcommand{\boldx}{\ensuremath{\mathbf{x}}}
\newcommand{\boldy}{\ensuremath{\mathbf{y}}}
\newcommand{\boldz}{\ensuremath{\mathbf{z}}}

\newcommand{\boldalpha}{\boldsymbol{\alpha}}
\newcommand{\boldbeta}{\boldsymbol{\beta}}

\newcommand{\boldomega}{\boldsymbol{\omega}}

\newcommand{\boldone}{\boldsymbol{1}}
\newcommand{\boldzero}{\boldsymbol{0}}

\newcommand{\Z}{\mathbb{Z}}
\newcommand{\R}{\mathbb{R}}
\newcommand{\Q}{\mathbb{Q}}
\newcommand{\C}{\mathbb{C}}
\newcommand{\CP}{\mathbb{CP}}
\newcommand{\V}{\mathcal{V}}

\newcommand{\imag}{\mathfrak{i}}

\newcommand{\inner}[2]{ \ensuremath{\left\langle {#1}, {#2}\right\rangle} }

\DeclareMathOperator{\evol}{vol}
\DeclareMathOperator{\mvol}{MV}
\DeclareMathOperator{\rank}{rank}
\DeclareMathOperator{\conv}{conv}
\DeclareMathOperator{\supp}{supp}
\DeclareMathOperator{\newt}{newt}

\newenvironment{smallbmatrix}{\left[\begin{smallmatrix}}{\end{smallmatrix}\right]}

\newsiamremark{remark}{Remark}
\newsiamremark{example}{Example}

\AtBeginDocument{%
\crefname{example}{example}{examples}
\Crefname{example}{Example}{Examples}
}

\title{
    A stratified polyhedral homotopy method for sampling positive-dimensional zero sets of polynomial systems%
    \thanks{Submitted to the editors DATE.
        \funding{%
            The author is supported, in part, by 
            the National Science Foundation under grant nos. 1923099 and 2318837,
            Auburn University at Montgomery through the Grant-in-Aid program,
            and the AMS-Simons Research Enhancement Grant for Primarily Undergraduate Institution Faculty
        }
    }
    \\[0.5ex]
    {\normalfont \footnotesize \sffamily  \textcolor{black}{In memory of Professor Tien-Yien Li}}    
}
\author{Tianran Chen%
    \thanks{Department of Mathematics,
        Auburn University at Montgomery,
        Montgomery, Alabama USA
        (\email{ti@nranchen.org})
    }
}

\begin{document}


\maketitle

\begin{abstract}
    Numerical algebraic geometry revolves around the study of
    solutions to polynomial systems via numerical methods.
    The polyhedral homotopy of Huber and Sturmfels
    for computing isolated solutions
    and the concept of witness sets
    as numerical representations of non-isolated solution components,
    put forth by Sommese and Wampler,
    are two pillars of this field.
    In this paper,
    we show that a modified polyhedral homotopy places the computation of isolated solutions and sample sets for non-isolated components into a single, unified framework.
    In certain cases, this method also leads to a natural decomposition of
    the BKK bound into a sum of its local contributions.
\end{abstract}



\section{Introduction}\label{sec: introduction}

Polynomial systems are at the heart of many scientific applications,
as computational problems often reduce to algebraic equations.
In recent decades, \emph{homotopy methods},
which offer unmatched scalability,
have revolutionized the computation of \emph{all} solutions to polynomial systems
\cite{ChenLi2015Homotopy,Li2003Numerical,SommeseWampler2005Numerical}. 
These methods work by continuously deforming a hard system into an easy one,
so that its solutions trace out smooth paths that
lead back to the target solutions,
allowing us to find all solutions efficiently by tracking these paths.

Among them, the \emph{polyhedral homotopy} of B. Huber and B. Sturmfels
\cite{HuberSturmfels1995Polyhedral}, developed in the 1990s,
is of particular importance,
as it optimally exploits combinatorial structures encoded in polynomial systems.
Around the same time, the seminal work by A. Sommese and C. Wampler
\cite{SommeseWampler1996Numerical} 
opened up a new frontier by allowing
non-isolated (a.k.a. positive-dimensional) solution sets
to be computed and manipulated as first-class objects through homotopy methods.
This idea matured in \cite{SommeseVerschelde2000Numerical}
and sparked a series of works
\cite{SommeseVersheldeWampler2001Numerical,SommeseVerscheldeWampler2001Monodromy,SommeseVerscheldeWampler2002Method,sommese_symmetric_2002}
that form the foundation of \emph{numerical algebraic geometry}.
In the ensuing years, polyhedral homotopy and numerical algebraic geometry
developed separately with minimal interactions.%
\footnote{%
    Notable exceptions include
    the polyhedral bootstrapping process used in the
    original ``cascade of homotopies'' in \cite{SommeseVerschelde2000Numerical,Verschelde1999PHCpack},
    the works from the group led by J. Verschelde
    \cite{AdrovicVerschelde2013Polyhedral,Verschelde2009Polyhedral}
    in which local Puiseux series representations of positive-dimensional zero sets
    are computed through polyhedral-like homotopy methods,
    as well as an unpublished program by T.-L. Lee
    for computing individual witness sets using \tech{HOM4PS-2.0}.}
The main goal of this paper is to demonstrate that
these two seemingly disconnected ideas are in fact
the two sides of the same coin.
Indeed, with a simple modification,
the polyhedral homotopy computes sample sets of
positive-dimensional solution components.

\subsection{Combining two homotopy approaches}
In this paper, we present a ``stratified'' version of the polyhedral homotopy
for sampling positive-dimensional solution sets defined by an
unmixed Laurent polynomial system with the following key advantages
\begin{enumerate}[leftmargin=7ex]
	\item 
		The number of paths is the Bernshtein-Kushnirenko-Khovanskii bound,
		whereas the complexity of the traditional approach is only bounded by
        the B\'ezout bounds.
	\item 
		This homotopy preserves the monomial structure,
		which is of particular importance in many problems originating from 
		science and engineering where the monomial structure imposes
		additional constraints that are crucial for specific applications.
	\item
		A single homotopy is used to sample components of all dimensions,
		including isolated solutions,
        and sample sets for non-isolated solution components
        are produced as by-products from the process of computing isolated solutions
        with minimal overhead.
\end{enumerate}

\subsection{Decomposition of BKK bound}

Bernshtein's theorem~\cite{Bernshtein1975Number}
states that for a Laurent polynomial system $F = (f_1,\ldots,f_n)$ in $n$ variables,
the number of isolated zeros in $(\C^*)^n = (\C \setminus \{0\})^n$
is bounded by the mixed volume of
the Newton polytopes $P_1,\ldots,P_n$ of the system.
It equals the normalized volume $n!\evol_n(P)$ if $P_1 = \cdots = P_n$
(i.e., the unmixed case established by Kushnirenko \cite{Kushnirenko1975Newton}).
This is the Bernshtein-Kushnirenko-Khovanskii (BKK) bound.
For generic coefficients, all zeros in $(\C^*)^n$ will be isolated,
and this bound will be exact.

However, if the zero set of $F$ in $(\C^*)^n$
contains positive-dimensional components,
then the number of isolated zeros in $(\C^*)^n$
must be less than the BKK bound \cite{Bernshtein1975Number}.
A natural question to ask is
if it is possible to decompose the BKK bound
as a sum of local contributions from each isolated zero
and the positive-dimensional components.

This mirrors the deep question of
decomposing the B\'ezout number
into local contributions from subvarieties
that is at the heart of intersection theory
\cite{Lazarsfeld1981Excess}.
The stratified polyhedral homotopy method
proposed in this paper will provide a homotopy-based answer.

\subsection{A motivating example}

We start with a simple motivating example.

\begin{example}\label{ex: running}
    Consider a trivial example of a polynomial system $F(x_1,x_2)$, given by
    \[
        \left\{
            \begin{aligned}
                (x_1^2 + x_2^2 - 9)(x_1 + x_2 - 3) &=
                x_1^3 + x_1^2 x_2 - 3x_1^2 + x_1 x_2^2 + x_2^3 - 3x_2^2 - 9x_1 - 9x_2 + 27 \\
                (x_1^2 + x_2^2 - 9)(x_1 - x_2 - 1) &=
                x_1^3 - x_1^2 x_2 - 1x_1^2 + x_1 x_2^2 - x_2^3 - 1x_2^2 - 9x_1 + 9x_2 + 9.
            \end{aligned}
        \right.
    \]
    Its complex zero set consists of two components:
    a 1-dimensional component $V_1$ defined by $x_1^2 + x_2^2 - 9 = 0$
    and a 0-dimensional (i.e., isolated) nonsingular component $V_0$
    at $\boldx^{(0)} = (x_1,x_2) = (2,1)$.
    When the standard polyhedral homotopy method (see \Cref{sec: polyhedral homotopy}) is applied,
    the nonsingular isolated zero $\boldx^{(0)}$ can be obtained.
    That is, the polyhedral homotopy defines solution paths,
    one of which reaches $\boldx^{(0)}$.
    With minor modifications, which the rest of this paper will detail,
    the polyhedral homotopy method can also produce a
    ``numerically well-behaved'' sample point from $V_1$.
    We consider the ``rank-1'' perturbation
    \[
        G(x_1,x_2) =
        \left\{
            \begin{aligned}
                c_{11}' x_1^3 + c_{12}' x_1^2 x_2 &+ c_{13}' x_1^2 + c_{14}' x_1 x_2^2 + c_{15}' x_2^3 + c_{16}' x_2^2 + c_{17}' x_1 + c_{18}' x_2 + c_{19}' \\
                c_{21}' x_1^3 + c_{22}' x_1^2 x_2 &+ c_{23}' x_1^2 + c_{24}' x_1 x_2^2 + c_{25}' x_2^3 + c_{26}' x_2^2 + c_{27}' x_1 + c_{28}' x_2 + c_{29}' ,
            \end{aligned}
        \right.
    \]
    which is derived from the target system $F$ by replacing the coefficient matrix with
    \[
        \left[
        \begin{smallmatrix}
            c_{11}' & c_{12}' & c_{13}' & c_{14}' & c_{15}' & c_{16}' & c_{17}' & c_{18}' & c_{19}' \\
            c_{21}' & c_{22}' & c_{23}' & c_{24}' & c_{25}' & c_{26}' & c_{27}' & c_{28}' & c_{29}' \\
        \end{smallmatrix}
        \right]
        =
        \left[
        \begin{smallmatrix}
            1 &  1 & -3 & 1 &  1 & -3 & -9 & -9 & 27 \\
            1 & -1 & -1 & 1 & -1 & -1 & -9 &  9 & 9 \\
        \end{smallmatrix}
        \right]
        +
        \left[
        \begin{smallmatrix}
            c_{11}^* & c_{12}^* & c_{13}^* & c_{14}^* & c_{15}^* & c_{16}^* & c_{17}^* & c_{18}^* & c_{19}^* \\
            c_{21}^* & c_{22}^* & c_{23}^* & c_{24}^* & c_{25}^* & c_{26}^* & c_{27}^* & c_{28}^* & c_{29}^* \\
        \end{smallmatrix}
        \right]
    \]
    \vspace{1.5ex}

\opencutleft
\renewcommand\windowpagestuff{\includegraphics[width=0.9\linewidth]{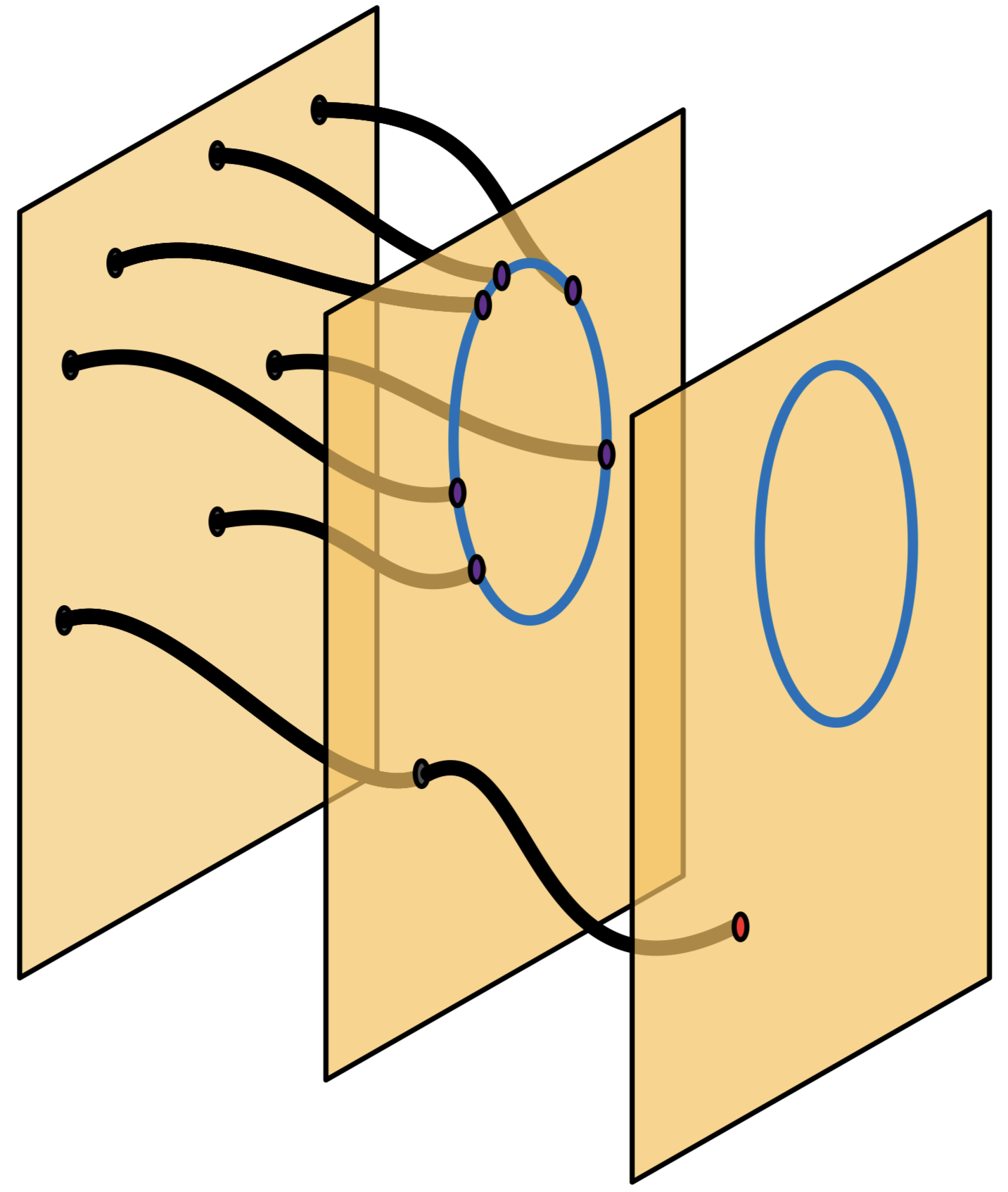}}
    \begin{cutout}{0}{5pt}{0.75\linewidth}{10}
        \noindent
        where $C^* = [ c_{ij}^* ]$ is a generic complex matrix of rank 1.
        That is, we modify the coefficient matrix with a generic rank-1 perturbation.
        Then among the isolated complex zeros of $G$,
        at least one is also contained in $V_1$ (the circle).
        These zeros depend on the choice of the generic perturbation $C^*$,
        but, regardless of the choice, this zero can serve as a ``numerically well-behaved''
        sample point of $V_1$ in the sense that it will be both
        a nonsingular zero of $G$ and a smooth point in $V_1$.
        We will define a modified polyhedral homotopy $H(x_1,x_2,s)$,
        which we will call a ``stratified'' polyhedral homotopy,
        such that $H(x_1,x_2,\frac{1}{2}) \equiv G(x_1,x_2)$
        and $H(x_1,x_2,0) \equiv F(x_1,x_2)$.
        Some of the solution paths defined by $H(x_1,x_2,s) = \boldzero$
        in $\C^2 \times [0,1]$ will reach sample points in $V_1$ at $s=\frac{1}{2}$
        and the isolated point $\boldx^{(0)}$ at the end point $s=0$.
        In other words, the sample point $\boldx^{(1)}$
        for the 1-dimensional solution component $V_1$
        is produced as a by-product of the process of computing the isolated solution $\boldx^{(0)}$.
        The cartoon illustration on the left shows the homotopy paths at
        $s=0,s=\frac{1}{2}$, and $s=1$,
        passing through sample points of $V_1$ (blue)
        and the isolated point $V_0$ (red).
    \end{cutout}
\end{example}

\subsection{Related works}

Our approach relates most closely to
the ``cascade of homotopies'' method developed by
A. Sommese and J. Verschelde \cite{SommeseVerschelde2000Numerical,Verschelde1999PHCpack}
for computing witness sets of positive-dimensional solution components,
the homotopy method developed by W. Zulehner \cite{Zulehner1989Solutions}
for finding one point on each connected component of
the complex zero set of a polynomial system,
and the stronger version developed by
D. Bates, D. Eklund, J. Hauenstein, and C. Peterson
\cite{BatesEklundHauensteinPeterson2021Excess}
for targeting the more refined structure known as isosingular sets.
A key advantage of our method is that,
for unmixed Laurent polynomial systems,
its complexity is linear in the BKK (Kushnirenko) bound of the system itself.

Similar to the ``twisted Chow form'' and ``toric perturbation''
developed by M. Rojas~\cite{Rojas1999Solving},
the proposed homotopy method also accelerates
the computation of positive-dimensional zero sets
by exploiting the combinatorial structure encoded in the
Newton polytope of the defining polynomial system.
The primary distinction is that our approach is homotopy-based,
whereas Rojas' is resultant-based.

Ultimately, what distinguishes this work is its simplicity.
With a minor tweak to the classical polyhedral homotopy,
the proposed method produces sample sets of positive-dimensional components
as a natural by-product from the process of computing isolated solutions.

\subsection{Organization}

In the rest of this paper, we will describe the construction
of this stratified polyhedral homotopy
and outline the theoretical underpinnings.
To be self-contained, \Cref{sec: preliminaries} will review
notations, concepts, and important theoretical ingredients.
\Cref{sec: standard unmixed case} develops the basic construction of a
stratified polyhedral homotopy method for sampling positive-dimensional solution sets
of an unmixed Laurent polynomial system.
General cases are considered in \Cref{sec: general cases}.
In \Cref{sec: BKK decomposition},
we explain how this homotopy method can, in certain cases,
produce a decomposition of the BKK bound
into local contributions from all components
as a by-product.
A few concrete examples are studied in \Cref{sec: examples}.
We conclude with a few remarks in \Cref{sec: conclusion}.
Technical details of a few well-known algorithms for bootstrapping the polyhedral homotopy method
are included in the appendix (\Cref{appendix: bootstrapping}) for completeness.

\section{Preliminaries}\label{sec: preliminaries}

Let $\zmat{n}{m}$ be the set of $n \times m$ integer matrices.
A matrix $U \in \zmat{n}{n}$ is \emph{unimodular} if
$U^{-1} \in \zmat{n}{n}$.
Each $A \in \zmat{n}{m}$ has a \emph{Smith Normal Form}:
there are unimodular 
$P \in \zmat{n}{n}, Q \in \zmat{m}{m}$ so that
$PAQ = \operatorname{diag}(d_1,\dots,d_r,0,\dots,0)$,
where $r = \op{rank} A$, and positive integers $d_1 \mid d_2 \mid \cdots \mid d_r$
are the \emph{invariant factors} of $A$.

For $\boldx=(x_1,\dots,x_n)$ and 
$\boldalpha = (\alpha_1,\dots,\alpha_n)^\top \in \Z^n$,
$
    \boldx^{\boldalpha} =
    x_1^{\alpha_1} \, \cdots \, x_n^{\alpha_n}
$
is a \emph{Laurent monomial}.
Similarly, for $A = \begin{bmatrix} \boldalpha^{(1)} & \cdots & \boldalpha^{(m)} \end{bmatrix} \in M_{n \times m}(\Z)$
the notation 
$
    \boldx^A =
    ( \boldx^{\boldalpha^{(1)}}, \dots, \boldx^{\boldalpha^{(m)}} )
$
describes a system of Laurent monomials.
It is natural to restrict the domain to
the \emph{algebraic torus} $(\C^*)^n = (\C \setminus \{0\})^n$,
which has a natural group structure given by componentwise multiplication.
A matrix $A \in \zmat{n}{m}$ induces a group homomorphism 
$\boldx \mapsto \boldx^A$ from $(\C^*)^n$ to $(\C^*)^m$,
which is also complex holomorphic.
If $A \in \zmat{n}{n}$ is unimodular,
then the map $\boldx \mapsto \boldx^A$ 
is an automorphism of the group $(\C^*)^n$, and it is also a bi-holomorphic map.

A \emph{Laurent binomial} in $\boldx = (x_1,\ldots,x_n)$
is an expression of the form
$c_1 \boldx^{\boldalpha} + c_2 \boldx^{\boldbeta}$
where
$\boldalpha,\boldbeta \in \Z^n$, and $c_1,c_2 \in \C^*$.
Without altering its zero set in $(\C^*)^n$, the equation
$c_1 \boldx^{\boldalpha} + c_2 \boldx^{\boldbeta} = 0$ can be rewritten as
$\boldx^{\boldalpha - \boldbeta} = -c_2/c_1$.
A \emph{Laurent binomial system} is a system of the form
$(\boldx^{\bolda^1}, \ldots, \boldx^{\bolda^m}) = (b_1,\dots,b_m)$
where $\bolda^i \in \Z^n$ and $b_i \in \C^*$ for $i=1,\dots,m$.
Using the matrix exponent notation, it can be written as
$\boldx^A = \boldb$
where the integer matrix $A \in M_{n \times m}(\Z)$ collects the exponents
and the row vector $\boldb \in (\C^*)^m$
collects all the constants.
\begin{lemma}
    \label{lem:binomial-roots}
    For a matrix $A \in \zmat{n}{n}$ and any $\boldb \in (\C^*)^n$,
    all isolated solutions of Laurent binomial system $\boldx^A = \boldb$ 
    are nonsingular, and the total number is
    $| \det A |$.
\end{lemma}

A \emph{Laurent polynomial} is a linear combination of Laurent monomials, i.e.,
an expression of the form $f = \sum_{k = 1}^m c_k \boldx^{\boldalpha^{(k)}}$
where each $c_k \in \C^*$ and $\boldalpha^{(k)} \in \Z^n$.
Here, the set $\supp(f) := \{ \boldalpha^1,\dots,\boldalpha^m \} \subset \Z^n$
is known as the \emph{support} of $f$.
Its convex hull $\newt(f) := \conv(\supp(f))$
is the \emph{Newton polytope} of $f$.
A \emph{Laurent polynomial system} is a system $F = (f_1,\dots,f_q)$
of Laurent polynomials in $n$ variables.
Its common zero sets in $(\C^*)^n$ and $\C^n$ 
are denoted by $\V^*(F)$ and $\V(F)$, respectively.
They are equipped with rich structures of 
\emph{very affine} and \emph{affine} varieties, respectively.
If nonempty, they are composed of \emph{irreducible components},
each with a well-defined dimension.
The unions of their $d$-dimensional components
are denoted by $\V_d^*(F)$ and $\V_d(F)$, respectively.
Kushnirenko's Theorem and Bernshtein's First Theorem
provide us the exact formulae for the maximum number of points in $\V_0^*(F)$.

\begin{theorem}[Kushnirenko \cite{Kushnirenko1975Newton}]
    \label{thm:kushnirenko}
    For a Laurent polynomial system $F = (f_1,\dots,f_n)$ in
    $\boldx = (x_1,\dots,x_n)$ with identical support
    $S = \supp(f_i)$ for $i=1,\dots,n$,
    $|\V_0^*(F)| \le n! \evol_n (\conv(S))$.
\end{theorem}

\begin{theorem}[Bernshtein's First Theorem \cite{Bernshtein1975Number}]
    \label{thm:bernshtein-a}
    For a Laurent polynomial system $F = (f_1,\dots,f_n)$ in the variables
    $x_1,\dots,x_n$,
    $|\V_0^*(F)| \le \mvol (\newt(f_1),\dots,\newt(f_n))$.
\end{theorem}

Here, $\mvol(P_1,\ldots,P_n)$ is the \emph{mixed volume}
of the convex polytopes $P_1,\ldots,P_n$,
and it is defined to be the coefficient of the monomial $\lambda_1\cdots\lambda_n$
in the volume of the \emph{Minkowski sum}
$\lambda_1 P_1 + \cdots + \lambda_n P_n$,
which is a homogeneous polynomial in $\lambda_1,\ldots,\lambda_n$.
The upper bounds given by both theorems are \emph{sharp}
in the sense that they hold with equality
for generic coefficients.
They have since been called the
Bernshtein-Kushnirenko-Khovanskii (BKK) bounds.

In the following subsections,
we briefly review the four main theoretical ingredients
from which we will develop the stratified polyhedral homotopy method.
Our review is by no means comprehensive,
and we refer to standard texts in this field
for thorough exposition.

\subsection{Polyhedral homotopy}\label{sec: polyhedral homotopy}

In their seminal work \cite{HuberSturmfels1995Polyhedral},
Huber and Sturmfels introduced the \emph{polyhedral homotopy} method
for computing \emph{all} isolated $\C^*$-zeros of Laurent polynomial systems
that can optimally exploit their monomial structure.\footnote{%
    In a parallel development, 
    a recursive homotopy method that can also
    take advantage of the Newton polytope structure
    to solve Laurent polynomial systems
    was proposed by J. Verschelde, P. Verlinden, and R. Cools
    around the same time \cite{VerscheldeVerlindenCools1994Homotopies}.
    This recursive homotopy method has also been referred to
    as \emph{polyhedral homotopy} in some papers.
    The present paper, however, only focuses on extending the
    polyhedral homotopy method of B. Huber and B. Sturmfels \cite{HuberSturmfels1995Polyhedral},
    which will be referred to simply as \emph{the polyhedral homotopy} hereafter.
}

For a \emph{square} Laurent system $F = (f_1, \ldots, f_n)$
in $\boldx = (x_1, \ldots, x_n)$ given by
\[
    f_i(\boldx) =
    \sum_{\bolda \in S_i}
        c_{i,\bolda}
        \boldx^{\bolda},
    \quad\text{for } i = 1, \dots, n,
\]
we select generic coefficients $c_{i,\bolda}^*$
for each pair of $i \in \{1,\ldots,n\}$ and $\bolda \in S_i$
and \emph{lifting functions} $\omega_i : S_i \to \Q^+$
with generic images for $i=1,\ldots,n$.
Among many variations,
the numerically stable formulation for
the polyhedral homotopy of Huber and Sturmfels
can be described as the homotopy function
$H = (h_1, \dots, h_n) : (\C^*)^n \times [0,1]^2 \to \C^n$
given by
\begin{equation}\label{equ: 2-step polyhedral exponential}
    h_i(\boldx, t_0, t_1) =
    \sum_{\bolda \in S_i}
        [t_1 c^*_{i,\bolda} + (1-t_1)c_{i,\bolda}]
        \boldx^{\bolda}
        e^{-M \omega_i(\bolda) t_0}
    \quad
    \text{for } i =  1, \ldots, n,
\end{equation}
where $M \in \R^+$ is determined by the
Newton polytopes of $h_1,\ldots,h_n$.
This numerically stable variant is different from the original formulation
by Huber and Sturmfels \cite{HuberSturmfels1995Polyhedral}
and was proposed by S. Kim and M. Kojima \cite{KimKojima2004Numerical}
and, independently, by T.-L. Lee, T.-Y. Li, and C.-H. Tsai
\cite{LeeLiTsai2008HOM4PS-2.0}.
This formulation will be referred to as
\emph{the classical polyhedral homotopy}.

Clearly, $H$ is continuous and $H(\boldx,0,0) \equiv F(\boldx)$.
Moreover, along any given smooth path $(t_0(s),t_1(s))$
in the parameter space $(0,1)^2$, under the genericity assumption,
the isolated $\C^*$-zeros of $H(\boldx,t_0,t_1)$ also vary smoothly
and form ``solution paths''.
The limit points of these solution paths as $(t_0,t_1) \to (0,0)$
reach \emph{all} isolated $\C^*$-zeros of $F$.
The starting points of these solution paths
can be computed by solving a series of Laurent binomial systems.
These binomial systems are, in turn, derived from a process known as
mixed cell computation.

Once these starting points are obtained,
the solution paths can be tracked
via numerical ``path tracking'' algorithms
\cite{bates_efficient_2011},
toward their end points,
which include all isolated $\C^*$-zeros of the target system $F$.
This is implemented in several software packages
\cite{BreidingTimme2017HomotopyContinuationjl,ChenLeeLi2014Hom4PS-3,gunji_phompolyhedral_2004,LeeLiTsai2008HOM4PS-2.0,Verschelde1999PHCpack}.

There is some flexibility in choosing the
parameter path $(t_0(s),t_1(s))$.
One choice that is widely adopted in recent implementations
\cite{ChenLeeLi2014Hom4PS-3,gunji_phompolyhedral_2004,LeeLiTsai2008HOM4PS-2.0}
is the path $(t_0(s),t_1(s)) = (s,s)$.
In contrast, the ``2-step'' procedure takes the piecewise-linear path
$(1,1) \to (0,1) \to (0,0)$.

\subsection{Parameter homotopy}\label{sec: parameter homotopy}

The smoothness of the solution paths defined by
the homotopy \eqref{equ: 2-step polyhedral exponential}
over a parameter path $\boldt(s)$
and their ability to reach \emph{all} isolated $\C^*$-zeros of the target system $F$
are the key features that make this homotopy method practical.
Indeed, much of the work in the field of numerical homotopy continuation methods
is devoted to the rigorous proof of these two properties
(nicknamed ``smoothness'' and ``accessibility'' properties
in Ref.~\cite{LiSauerYorke1987Random})
for various homotopy constructions.
One general result that will be referenced repeatedly
is the Parameter Homotopy Theorem of
A. Morgan and A. Sommese \cite{MorganSommese1989Coefficient}
for homotopy constructions of the form $H(\boldx,t) = F(\boldx;\boldp(t))$
where the coefficients of a polynomial system $F$ are polynomial functions
in complex parameters $\boldp = (p_1,\ldots,p_m)$.

\begin{theorem}[Parameter homotopy (\cite{SommeseWampler2005Numerical} Theorem 7.1.1, \cite{MorganSommese1989Coefficient})]
    \label{thm: parameter homotopy}
    Let $F(\boldx;\boldp)$ be a system of $n$ polynomials in the $n$ variables
    $\boldx = (x_1,\dots,x_n)$ and $m$ parameters $\boldp = (p_1,\dots,p_m)$,
    and let $\mathcal{N}(\boldp)$ be the number of (isolated) nonsingular zeros of $F(\boldx;\boldp)$
    in $\C^n$ for a given $\boldp$.
    Then,
    \begin{enumerate}
        \item $\mathcal{N}(\boldp)$ is finite, and it is the same, say $\mathcal{N}$,
            for almost all $\boldp \in \C^m$;
        \item For all $\boldp \in \C^m$, $\mathcal{N}(\boldp) \le \mathcal{N}$;
        \item The subset of $\C^m$ where $\mathcal{N}(\boldp) = \mathcal{N}$ is Zariski open (and nonempty),
            i.e., the exceptional set $P^* = \{ \boldp \in \C^m \mid \mathcal{N}(\boldp) < \mathcal{N} \}$
            is an affine algebraic set contained within an algebraic set of dimension $m-1$.
        \item The homotopy $F(\boldx;\boldp(t))=0$ with an analytic function 
            $\boldp(t) : [0,1] \to \C^m \setminus P^*$
            has $\mathcal{N}$ continuous and nonsingular solution paths;
        \item As $t \to 0$, the limits of the solution paths of the homotopy $F(\boldx;\boldp(t))=0$
            with $\boldp(t) : (0,1] \to \C^m \setminus P^*$ include all the
            (isolated) nonsingular zeros of $F(\boldx; \boldp(0)) = 0$ in $\C^n$.
    \end{enumerate}
\end{theorem}

The variation with $\boldp$ being the coefficients in $F$,
including constants, was also discovered by
T.-Y. Li, T. Sauer, and J. Yorke \cite{LiSauerYorke1989Cheater},
which led to the extension of the BKK bound
\cite{LiWang1996BKK,RojasWang1996Counting}.

\subsection{Positive-dimensional zero sets and witness sets}\label{sec: positive dimensional}

In their pioneering work~\cite{SommeseWampler1996Numerical}, 
A. Sommese and C. Wampler coined the term \emph{numerical algebraic geometry}
and thus kickstarted the study of
positive-dimensional solution sets of polynomial systems (\emph{algebraic sets}),
via numerical homotopy methods.
(See Refs.~\cite{BatesHauensteinSommeseWampler2013Numerically,HauensteinSommese2017What,SommeseVerscheldeWampler2005Introduction}
for a practical manual, an accessible survey, a broad overview, respectively).
One of the fundamental building blocks in the field is
the concept of ``linear slices''.
A linear slice of a solution set
is its intersection with an affine subspace,
which can help reveal important structural information
about the solution set itself.

\begin{theorem}[Linear Slicing (\cite{SommeseWampler2005Numerical} Theorem 13.2.1)]\label{thm:linear slicing}
    Let $V \subset \C^{m}$ be a pure $d$-dimensional algebraic set.
    There is a Zariski open dense $U \subset \CP^m$ such that
    for $\boldc \in U$ and $L = \V(\mathcal{L}(\boldz;\boldc))$,
    \begin{enumerate}
        \item if $d=0$, then $L \cap V$ is empty;
        \item if $d>0$, then $L \cap V$ is nonempty and $(d-1)$-dimensional;
        \item if $d>1$ and $V$ is irreducible, then $L \cap V$ is irreducible.
    \end{enumerate}
\end{theorem}

Here, the defining equations of hyperplanes in $\C^m$
are parametrized by points in the complex projective space $\CP^m$,
through the map
\[
    [c_0 : c_1 : \cdots : c_m]
    \mapsto
    \mathcal{L}(z_1,\ldots,z_m; c_0,c_1,\ldots,c_m) :=
    c_0 + c_1 z_1 + \cdots + c_m z_m.
\]
The stronger version needed in this paper allows for systems of linear polynomials
used as linear slicing equations,
which we restate here.

\begin{proposition}[\cite{SommeseWampler2005Numerical} Theorem 13.2.2 and Lemma 13.2.3]
    \label{thm:linear slicing system}
    Let $V \subset \C^{m}$ be a pure $d$-dimensional affine algebraic set with $d \ge 1$.
    There is a Zariski open dense subset $U \subset (\C^{(m+1)})^k$ such that
    for $(\boldc_1^*,\dots\boldc_k^*) \in U$ and 
    $L = \V(\mathcal{L}(\boldz;\boldc_1^*,\dots,\boldc_k^*))$,
    \begin{enumerate}
        \item if $d < k$, then $L \cap V$ is empty;
        \item if $d > k$, then $L \cap V$ is nonempty and positive-dimensional;
        \item if $d = k$, then $L \cap V$ is nonempty and 0-dimensional.
    \end{enumerate}
    Moreover, if $V$ is a component of the zero set of a polynomial system $F$
    of multiplicity 1, then $L \cap V$ is a component of
    $\V(F,\mathcal{L}(\boldz;\boldc_1^*,\dots,\boldc_k^*))$ of multiplicity 1.
\end{proposition}

The last case in the list above is of particular importance,
and it leads to the concept of \emph{witness set}
\cite{SommeseVerschelde2000Numerical,SommeseWampler1996Numerical}
that has its theoretical underpinning in the rich classical study
of the connections between algebraic sets and their linear sections
\cite{BeltramettiSommese1995Adjunction}.

\begin{remark}
    The linear slices in this proposition are simply parametrized by
    $k$-tuples of complex vectors $C = (\boldc_1^*,\ldots,\boldc_k^*)$.
    As noted in Ref.~\cite{SommeseWampler2005Numerical},
    this is a rather coarse parametrization
    since the image of $C$ under any nonsingular linear transformation
    would result in the same linear slicing.
    The much more natural parameter space is the Grassmannian $\operatorname{Gr}(k,\C^m)$.
    This distinction, however, is not important in our discussion,
    and we will prefer the parametrization using $k$-tuples of complex vectors
    since they can be chosen at random directly.
\end{remark}

\subsection{Randomization}\label{sec: randomization}

The final ingredient is the ``randomization'' process.
For a system $F$ of $q$ Laurent polynomials
and a $k \times q$ matrix $\Lambda$,
every zero of $F$ is, of course, a zero of $\Lambda \cdot F$,
if $F$ is considered as a column vector.
The following result provides the complete description
of the connection between the zero sets of $F$ and $\Lambda \cdot F$, respectively,
for generic choices of $\Lambda$.

\begin{theorem}[\cite{SommeseWampler2005Numerical} Theorem 13.5.1]\label{thm:randomization}
    Let $F = (f_1,\dots,f_q)$ be a system of polynomials on $\C^n$.
    Assume $V \subset \C^n$ is an irreducible affine algebraic set.
    Then there is a nonempty Zariski open set $U$ of $k \times q$ matrices
    such that for all $\Lambda \in U$,
    \begin{enumerate}
        \item if $\dim V > n - k$, then $V$ is an irreducible component of $\V(F)$
            if and only if it is an irreducible component of $\V(\Lambda \cdot F)$;
        \item if $\dim V = n - k$, then $V$ is an irreducible component of $\V(F)$
            implies that $V$ is also an irreducible component of $\V(\Lambda \cdot F)$;
        \item if $V$ is an irreducible component of $\V(F)$,
            its multiplicity as a solution component of $\Lambda \cdot F(\boldx) = \boldzero$
            is greater than or equal to its multiplicity as a solution component of
            $F(\boldx) = \boldzero$, with equality if either multiplicity is 1.
    \end{enumerate}
\end{theorem}

This produces a particularly useful preprocessing step for solving
overdetermined polynomial systems.
Any system of $q$ polynomials $F = (f_1,\ldots,f_q)$ in $n$ variables
with $q > n$ can be converted into a \emph{square} system
$\Lambda \cdot F$ of $n$ polynomials in $n$ variables
through a generic $n \times q$ matrix $\Lambda$.
Every zero of $F$ will be a zero of $\Lambda \cdot F$.

\section{Stratified polyhedral homotopy for standard unmixed cases}\label{sec: standard unmixed case}

Our main idea can be summarized in one sentence:
By modifying the polyhedral homotopy method to move the coefficients along a path
that systematically traverses the strata of the coefficient space,
we reveal sample sets of reduced positive-dimensional components one by one.

This section elaborates on this idea:
We define a homotopy algorithm,
in the spirit of the cascade method \cite{SommeseVerschelde2000Numerical}
by A. Sommese and J. Verschelde,
for numerically sampling reduced\footnote{%
    A \emph{reduced} (or \emph{generically reduced}) irreducible component of $\V^*(F)$
    is simply an irreducible component of multiplicity 1.
    At almost all points,
    the nullity of the Jacobian matrix $DF$
    equals its dimension.
}
irreducible components of all dimensions
of the zero set of a Laurent system $F = (f_1,\ldots,f_q)$.

The goal is to construct a homotopy function $H(\boldx,s)$
such that its zero set
$\{ (\boldx,s) \in (\C^*)^n \times (0,1] \mid H(\boldx,s) = \boldzero \}$
consists of piecewise smooth solution paths
that will pass through finite ``sample sets''
$V_n, V_{n-1}, \ldots, V_1, V_0$
with $V_d \subset \V^*_d(F)$
and $V_d = \varnothing$ if and only if $V^*_d(F) = \varnothing$
for $d=n,n-1,\ldots,0$.
Moreover, for each reduced irreducible component of $\V^*_d(F)$,
$V_d$ contains at least one nonsingular point of that component. 
In other words, the homotopy $H$ defines homotopy paths
that can sample every reduced irreducible component of $\V^*(F)$.

For simplicity, we focus on unmixed Laurent systems.
Recall that a Laurent system $F = (f_1,\dots,f_q)$ is \emph{unmixed}
if the supports $\supp(f_i)$, for $i=1,\ldots,q$, are all identical.
In this case, this common support is denoted $\supp(F)$.
We can express such an unmixed Laurent polynomial system
in $n$ variables in the compact notation
\begin{equation}\label{equ:unmixed-sys}
    F(x_1,\dots,x_n) = F(\boldx) =
    \left\{
        \begin{aligned}
            f_1(\boldx) &= \boldc_1 \cdot \boldx^A \\
            &\vdotswithin{=} \\
            f_q(\boldx) &= \boldc_q \cdot \boldx^A,
        \end{aligned}
    \right.
\end{equation}
where the support matrix $A = [ \, \bolda_1 \; \cdots \; \bolda_m \,] \in \zmat{n}{m}$, 
with $m = |\supp(F)| > 0$, collects the exponent vectors in $\supp(F)$ as columns,
$\boldc_k$'s are row vectors collecting corresponding coefficients,
and $\boldc_k \cdot \boldx^A$ denotes the dot product between the two row vectors.
To further simplify our constructions,
we restrict our attention to systems in a ``standard form''.
More general cases will be discussed in \Cref{sec: general cases}.

\begin{definition}[Standard form]\label{def: standard form}
    The unmixed Laurent polynomial system in \eqref{equ:unmixed-sys} is said to be
    in \term{standard form} if
    the support matrix $A \in \zmat{n}{m}$ has the following properties
    \begin{enumerate}
        \item $m > n + 1$ \label{cond: nonbinomial};
        \item $A$ has a zero column \label{cond: constant term};
        \item $A$ has full row rank \label{cond: full rank};
        \item The invariant factors of $A$ are $1$ \label{cond: torsion free}.
    \end{enumerate}
\end{definition}

These conditions can be assumed without losing much generality:
Condition \ref{cond: nonbinomial} simply eliminates simpler systems
for which the proposed method would be unnecessary.
Indeed, if $m \le n+1$, then the $\C^*$-zero set of $F$
is either empty or defined by binomials.
Condition \ref{cond: constant term} is the requirement that
each Laurent polynomial has nonzero constant term,
and it can be satisfied by multiplying each polynomial by a Laurent monomial
without altering the $\C^*$-zero set.
Condition \ref{cond: full rank} ensures that
there is no nontrivial toric action
on the $\C^*$-zero set when generic coefficients are used.
If $r := \rank(A) < n$, then every $\C^*$-zero $\boldx$ of $F$
belongs to a toric orbit of zeros parametrized by a $\C^*$-valued function
$\boldt \mapsto \boldx \circ \boldt^{\boldv}$ defined on $(\C^*)^{n-r}$,
where $\boldv$ is a primitive generator of the left kernel of $A$.
In that case, the $\C^*$-zero set of $F$
can be projected down to $(\C^*)^r$
so that it is defined by an unmixed Laurent system that satisfies this condition.
Finally, condition~\ref{cond: torsion free},
i.e. the torsion-free condition,
greatly simplifies our discussions, and
\Cref{sec: lattice reduction} describes the procedure
that reduces a general case to a torsion-free one.

\subsection{Laurent polynomial systems as linear slices}

We demonstrate that from the right angle,
the seemingly independent approaches of the polyhedral homotopy and
the linear slicing method share an underlying numerical reality,
a perspective we make precise in this lemma:

\begin{lemma}\label{lem:toric-linear}
    Let $A \in \zmat{n}{m}$ be the support matrix of the 
    unmixed system \eqref{equ:unmixed-sys} in standard form.
    Then there exists a matrix $B \in \zmat{m}{m-n}$ for which $\V^*(F)$
    is the preimage, under the bi-holomorphic map $\phi_A(\boldx) = \boldx^A$, 
    of the $\C^*$-zero set of the Laurent system
    \begin{equation*}
        G(z_1,\dots,z_m) = G(\boldz) =
        \left\{
            \begin{aligned}
                \boldz^B - \boldone &= \boldzero \\
                \boldc_i \cdot \boldz &= 0 & \text{for } i = 1,\ldots,q \\
            \end{aligned}
        \right.
    \end{equation*}
    where $\boldc_1,\dots,\boldc_q$ are the coefficient vectors of the
    original polynomial system \eqref{equ:unmixed-sys}.
\end{lemma}

This is the basic setup for the ``$A$-philosophy''
for Laurent systems consolidated in the classical text by
I. Gel'fand, M. Kapranov, and A. Zelevinsky \cite{GelfandKapranovZelevinsky1994Discriminants}.
We include an elementary and constructive proof for later reference.

\begin{proof}
    Under the assumption
    that $A$ is of full row rank and
    has invariant factors are all 1's,
    there are unimodular matrices
    $P \in \zmat{n}{n},Q \in \zmat{m}{m}$
    such that
    $PAQ = \begin{bmatrix} I_n & \boldzero_{n \times k} \end{bmatrix}$,
    where $k = m - n > 0$.
    Let $B \in \zmat{m}{k}$ be the rightmost $k$ columns of $Q$,
    which span $\ker A$, 
    and let $C \in \zmat{k}{m}$ be the bottommost $k$ rows of $Q^{-1}$.
    Then $C B = I_k$, and hence
    \[
        \begin{bmatrix}
            C \\ A
        \end{bmatrix}
        B =
        \begin{bmatrix}
            C B \\
            A B
        \end{bmatrix}
        =
        \begin{bmatrix}
            I \\
            \boldzero
        \end{bmatrix}.
    \]
    Our assumptions ensure that
    $\begin{smallbmatrix} C \\ A \end{smallbmatrix}$ is unimodular.
    Therefore, $\begin{smallbmatrix} C \\ A \end{smallbmatrix}^{-1}
    \in \zmat{m}{m}$.

    Let $T = \V^*(\boldz^B - \boldone) \subset (\C^*)^m$.
    We shall construct a bi-holomorphic map between points in $\V^*(F)$
    and points in a linear slice of $T$.
    Consider the map $\phi_A : (\C^*)^n \to (\C^*)^m$ given by $\phi_A(\boldx) := \boldx^A$.
    For any $\boldx \in (\C^*)^n$,
    $
        (\phi_A(\boldx))^B =
        \boldx^{AB} =
        \boldx^{\boldzero} =
        \boldone,
    $
    and thus $\phi_A(\boldx) \in T$.
    That is, $\phi_A((\C^*)^n) \subseteq T$.
    It remains to show that $\phi_A$ is a bi-holomorphic map onto $T$.

    Define $\psi : T \to (\C^*)^m$ given by
    $\psi(\boldz) =
        \boldz^{\left[\begin{smallmatrix}C\\ A\end{smallmatrix}\right]^{-1}}
    $.
    For any $\boldz \in T$, write $\psi(\boldz)$ as $[\boldy\;\boldx]$ with
    $\boldx \in (\C^*)^n$ and $\boldy \in (\C^*)^k$,
    then by construction $\boldz = \psi(\boldz)^{\begin{smallbmatrix} C \\ A \end{smallbmatrix}}$,
    and hence
    \[
        \boldone =
        \boldz^B =
        (\psi(\boldz))^{\begin{smallbmatrix} C \\ A \end{smallbmatrix} B} =
        \begin{bmatrix}
            \boldy & \boldx
        \end{bmatrix}^{\begin{smallbmatrix} I \\ \boldzero \end{smallbmatrix}}
        = \boldy \circ \boldone = \boldy.
    \]
    Therefore,
    \[
        \boldz = \psi(\boldz)^{\begin{smallbmatrix}C \\ A\end{smallbmatrix}}
        = \begin{bmatrix}
        \boldone & \boldx
        \end{bmatrix}^{\begin{smallbmatrix}C \\ A\end{smallbmatrix}}
        = \boldone^C \circ \boldx^A
        = \boldx^A.
    \]
    Let $\pi : (\C^*)^m \to (\C^*)^n$ be the projection to the last $n$ coordinates,
    then for any $\boldz \in T$,
    \[
        \phi_A(\pi(\psi(\boldz))) = \phi(\boldx) = \boldx^A = \boldz.
    \]
    Conversely, for any $\boldx \in (\C^*)^n$,
    \[
        \pi(\psi(\phi_A(\boldx))) =
        \pi(\psi(\boldx^A)) =
        \pi \left(
            \boldx^{A \begin{smallbmatrix} C \\ A \end{smallbmatrix}^{-1}}
        \right) =
        \pi \left(
            \boldx^{\begin{smallbmatrix} \boldzero & I \end{smallbmatrix}}
        \right) =
        \pi \left(
            \begin{bmatrix}
                \boldone & \boldx
            \end{bmatrix}
        \right) =
        \boldx.
    \]
    Therefore, the composition $\pi \circ \psi : T \to (\C^*)^n$ is the
    inverse of the restriction of $\phi_A$ onto $T$ as shown in the following commutative diagram:
    \begin{center}
        \begin{tikzpicture}[every node/.style={midway}]
        \matrix[column sep={4em,between origins},
        row sep={2em}] at (0,0)
        {
            \node(Cm)   {$(\C^*)^m$}  ; & \node(T) {$T$}; \\
            \node(Cn) {$(\C^*)^n$};                   \\
        };
        \draw[<-] (Cn) -- (Cm) node[anchor=east]  {$\pi$};
        \draw[->] (Cn) -- (T)  node[anchor=north] {$\phi_A$};
        \draw[->] (T)  -- (Cm) node[anchor=south] {$\psi$};
        \end{tikzpicture}
    \end{center}
    Moreover, since both $\phi_A$ and $\pi \circ \psi$ are given by Laurent monomial maps, 
    the restriction $\phi : (\C^*)^n \to T$ is a (bijective) bi-holomorphic map.
    Hence, we have the bi-holomorphic correspondence
    between the $\C^*$-zero sets: 
    \begin{align*}
        \boldc_i \cdot \boldx^A &= 0 \quad\text{for } i = 1,\ldots,q
        &&\Longleftrightarrow&
        &
        \left\{
        \begin{aligned}
            \boldz^B - \boldone &= \boldzero \\
            \boldc_i \cdot \boldz &= \boldzero & \text{for } i = 1,\ldots,q.
        \end{aligned}
        \right.
    \end{align*}
    as claimed.
\end{proof}

\subsection{Toric slicing formulation}

From the viewpoint of the above lemma,
irreducible components in $\V^*(F)$ can be sampled through toric versions
of linear slices.

\begin{definition}\label{def:toric slicing}
    Given an unmixed Laurent system,
    as given in \eqref{equ:unmixed-sys},
    a nonnegative integer $d \le n$,
    and vectors $\boldc^*_1,\dots,\boldc^*_d \in \C^m$,
    we define the \term{rank $d$ toric slicing system} to be
    \begin{equation}\label{equ:toric slicing}
        F^{(d)} (\boldx) =
        \begin{cases}
            \boldc_k   \cdot \boldx^A & \text{for } k = 1,\ldots,q \\
            \boldc_k^* \cdot \boldx^A & \text{for } k = 1,\ldots,d
        \end{cases}
    \end{equation}
    The set $\V^*_0(F^{(d)}) \subset \V^*(F)$ will be called a
    \term{rank $d$ sample set} of $F$.
\end{definition}

While this definition implicitly depends on the choice of 
$\boldc^*_1,\dots,\boldc^*_d \in \C^m$,
this dependence is of little interest here
as the choice is always assumed to be generic in our discussions.

\begin{lemma}
    Let $F$, given in~\eqref{equ:unmixed-sys}, be an unmixed Laurent polynomial system
    in standard form.
    If $\V^*_d(F)$ is nonempty and reduced for some nonnegative integer $d < n$,
    then there is a nonempty Zariski open and dense set $U \subseteq (\C^m)^d$
    such that for all $(\boldc^*_1,\dots,\boldc^*_d) \in U$, 
    $\V^*_0(F^{(d)})$ consists of finitely many nonsingular points,
    and all these points are in $\V^*_d(F)$.
\end{lemma}

\begin{proof}
    By~\Cref{lem:toric-linear}, there is a matrix $B \in M_{m \times (m-n)}(\Z)$
    such that $\V^*(F) \subset (\C^*)^n$ is bi-holomorphically equivalent to
    $V' = \V^*(\boldz^B - \boldone,\mathcal{L}(\boldz; \boldc_1,\ldots,\boldc_q)) \subset (\C^*)^m$.
    Then each $d$-dimensional irreducible component of $\V^*(F)$ corresponds to a
    $d$-dimensional irreducible component of $V'$.
    Under the same map, $\V^*(F^{(d)})$
    is equivalent to $\V^*(G^{(d)})$, where
    \begin{equation*}
        G^{(d)} (\mathbf{z}) = 
        \begin{cases}
            \mathbf{z}^B - \boldone \\
            \mathbf{c}_k   \cdot \mathbf{z} & \text{for } k = 1,\ldots,q \\
            \mathbf{c}_k^* \cdot \mathbf{z} & \text{for } k = 1,\ldots,d.
        \end{cases}
    \end{equation*}
    Note that $\V^*(G^{(d)})$ is precisely the linear slice of $\V^*(G)$
    with respect to $\mathcal{L}(\boldz; \boldc^*_1, \dots, \boldc^*_d)$.
    By the Linear Slicing Theorem (\Cref{thm:linear slicing,thm:linear slicing system}),
    $\V^*_0(G^{(d)}) \subset (\C^*)^m$ consists of nonsingular points
    and is contained in the $d$-dimensional components of $V'$.
\end{proof}

More generally, if the requirement for $\V^*_d(F)$ to be reduced is dropped,
$\V^*_0(F^{(d)})$ may contain singular points,
i.e., points where $\rank DF^{(d)} < n$.
Yet, by restriction, the above constructions can still be applied to
each individual reduced irreducible component of $\V^*_d(F)$:

\begin{corollary}
    Let $V \ne \varnothing$ be a reduced irreducible component of $\V^*_d(F)$,
    then there is a nonempty Zariski open and dense set $U \subseteq \C^{m \times d}$
    such that for all $(\boldc^*_1,\dots,\boldc^*_d) \in U$, 
    $\V^*_0(F^{(d)}) \cap V$ is nonempty, 
    and it consists of finitely many nonsingular points in $V$.
\end{corollary}

This justifies that a rank $d$ sample set of $F$
does indeed sample all reduced $d$-dimensional components of $\V^*(F)$.
The following subsections aim to set up an efficient homotopy algorithm
for computing each sample set as a direct extension of the
classical polyhedral homotopy.
Indeed, \emph{all} sample sets will be connected
through solution paths defined by a single homotopy.

\subsection{Square system formulation}\label{sec: square formulation}

The toric slicing system \eqref{equ:toric slicing} (in \Cref{def:toric slicing})
is a system of $q + d$ Laurent polynomials in $n$ variables.
While it is possible to study it directly,
it is much more convenient to turn it into a square system.
In the following, let $r = n - q$.
As noted in \Cref{sec: randomization},
we only need to focus on cases where $n \ge q$, and hence $r \ge 0$.
From \Cref{thm:randomization}, we can derive the following result.

\begin{lemma}
    If $d > r$,
    let $\Lambda = [\lambda_{i,j}]$ be a $n \times (d-r)$ matrix
    and consider the square system
    \begin{equation*}
        F^{(d)}_{\square} (x_1,\dots,x_n) =
        F^{(d)}_{\square} (\boldx) =
        \begin{cases}
            \left(\boldc_i   + \sum_{k=r+1}^{d} \lambda_{i,  k} \boldc^*_k \right) \cdot \boldx^A
            &\text{for } i = 1,\ldots,q \\[1ex]
            \left(\boldc_i^* + \sum_{k=r+1}^{d} \lambda_{q+i,k} \boldc^*_k \right) \cdot \boldx^A
            &\text{for } i = 1,\ldots,r
        \end{cases}
    \end{equation*}
    For generic choices of $\Lambda$,
    all  isolated points in $\V^*(F^{(d)})$ are
    also isolated points in $\V^*(F^{(d)}_{\square})$.
    Furthermore, $\V^*(F^{(d)})$ and $\V^*(F^{(d)}_{\square})$ have the 
    same set of positive-dimensional components.
\end{lemma}

This transformation turns a toric slicing system into a square system
while capturing all the $\C^*$-zeros.
It is possible for this transformation to introduce extraneous zeros,
i.e., isolated points that are in $\V^*(F^{(d)}_\square) \setminus \V^*(F^{(d)})$.
Therefore, $\V^*(F^{(d)}_\square)$ will be referred to as a
\term{rank $d$ sample superset} of $F$.
However, as we shall see, these extraneous points are far from useless.
On the contrary, they are crucial in our construction of homotopy paths
that will chain all sample sets together.
But first, an explanation of the term ``rank'' is in order.

\begin{remark}[The meaning of ``rank'']\label{rmk: square system}
    In the special case of $n = q$, i.e., $F$ being a square system,
    the corresponding square system can be expressed concisely as
    \[
        F^{(d)}_{\square}(\boldx) =
        (C + \Lambda \, C^*) (\boldx^A)^\top
    \]
    where $C,C^*,\Lambda$ are complex matrices of sizes
    $n \times m$, $d \times m$, and $n \times d$, respectively, given by
    \begin{align*}
        C &= 
        \left[
        \begin{smallmatrix}
            \boldc_1 \\
            \vdots \\
            \boldc_n
        \end{smallmatrix}
        \right],
        &
        C^* &=
        \left[
        \begin{smallmatrix}
            \boldc^*_1 \\
            \vdots \\
            \boldc^*_d
        \end{smallmatrix}
        \right],
        &
        \Lambda &=
        \left[
        \begin{smallmatrix}
            \lambda_{11} & \cdots & \lambda_{1d} \\
            \vdots & \ddots & \vdots \\
            \lambda_{n1} & \cdots & \lambda_{nd}
        \end{smallmatrix}
        \right]
        .
    \end{align*}
    In this form, it is easy to see that
    $F^{(d)}_{\square}(\boldx)$ is exactly a perturbed version of the original system $F$
    in which the coefficient matrix $C$ is replaced by $C+\Lambda C^*$
    where $\Lambda C^*$ is a generic matrix of rank $d$.
    This interpretation justifies the usage of the term ``rank'' in
    ``rank $d$ sample superset''.

    This perspective suggests stratifying the coefficient space
    by the rank of the perturbation,
    where the original system is the rank $0$ stratum and
    generic systems form the rank $n$ stratum.
    The central idea is to construct a single polyhedral homotopy that
    moves downward through these strata to compute all sample sets at once.
\end{remark}

\subsection{Stratified polyhedral homotopy}\label{sec: stratified}

We now construct the homotopy method that can compute the rank $d$ sample superset
$\V^*_0(F^{(d)}_\square)$, which contains the rank $d$ sample sets of $F$,
for each $d=1,\ldots,n$ using a single homotopy procedure.
The first component in this procedure is the natural connection between
consecutive sample supersets.

\begin{definition}
    For the $F^{(d)}_{\square}$ defined above,
    we define $H^{(d)} : \C^n \times \C \to \C^n$, given by
    \begin{equation}\label{equ:stratified}
        H^{(d)}_i(\boldx,t) = 
        \left\{
        \begin{aligned}
            \left(\boldc_i   + \sum_{k=r+1}^{d-1} \lambda_{ik}    \boldc^*_k + t \lambda_{i,d}   \boldc^*_d \right) &\cdot \boldx^A 
            \quad\text{for } i = 1,\ldots,q \\
            \left(\boldc_i^* + \sum_{k=r+1}^{d-1} \lambda_{q+i,k} \boldc^*_k + t \lambda_{q+i,d} \boldc^*_d \right) &\cdot \boldx^A
            \quad\text{for } i = 1,\ldots,r \\
        \end{aligned}
        \right.
    \end{equation}
\end{definition}

Clearly, $H^{(d)}(\boldx, 0) \equiv F^{(d-1)}_{\square}(\boldx)$ 
and      $H^{(d)}(\boldx, 1) \equiv F^{(d)  }_{\square}(\boldx)$.
Furthermore, by restricting $t$ to the real interval $[0,1]$,
we get a homotopy function between $F^{(d-1)}_{\square}$ and $F^{(d)}_{\square}$ since
$H^{(d)}$ is continuous in both $\boldx$ and $t$.
We shall show that the isolated $\C^*$-zeros of $H^{(d)}$ also move smoothly,
as $t$ goes from 1 to 0, forming smooth solution paths in $(\C^*)^n \times (0,1]$.

\begin{theorem}
    For generic $\boldc^*_k$'s and $\{\lambda_{i,j}\}$,
    the zero set of $H^{(d)}$ in $(\C^*)^n \times (0,1]$ consists of
    finitely many smooth solution paths
    emanating from the nonsingular points of $\V^*_0( F^{(d)}_\square )$ at $t=1$,
    and the set of their finite limit points as $t \to 0$
    contains \emph{all} points in $\V^*_0( F^{(d-1)}_\square )$.
\end{theorem}

\begin{proof}
    Define $G = (G_1,\ldots,G_n) : (\C^*)^n \times \C^m \to \C^n$ given by
    \begin{equation*}
        G_i(\boldx,\mathbf{p}) = \left\{
        \begin{aligned}
            \left(\boldc_i   + \sum_{k=r+1}^{d-1} \lambda_{ik}    \boldc^*_k + \lambda_{i,d}   \, \mathbf{p} \right) &\cdot \boldx^A
            \quad\text{for } i = 1,\ldots,q \\
            \left(\boldc_i^* + \sum_{k=r+1}^{d-1} \lambda_{q+i,k} \boldc^*_k + \lambda_{q+i,d} \, \mathbf{p} \right) &\cdot \boldx^A
            \quad\text{for } i = 1,\ldots,r
        \end{aligned}
        \right.
    \end{equation*}
    which represents a family of Laurent systems parametrized by
    $\mathbf{p} \in \C^m$ that contains $H^{(d)}(\boldx,t)$ for all $t$
    since $H^{(d)} (\boldx,t) = G(\boldx, t \, \boldc_d^*)$.
    By the Parameter Homotopy Theorem~(\Cref{thm: parameter homotopy}),
    for a generic $\mathbf{p} \in \C^m$,
    the total number of nonsingular points in $\V^*_0( G(\boldx, \mathbf{p}) )$ is finite,
    and it is the same number, say $\mathcal{N}$.
    The exceptional set $Q$ of the parameters for which the number of nonsingular points
    in $\V^*_0( G(\boldx, \mathbf{p}) )$ is less than $\mathcal{N}$
    is contained in a proper algebraic set.
    In particular, at $t=1$, $H^{(d)}(\boldx,1) = G(\boldx,\boldc^*_d)$,
    so for generic choices of $\boldc^*_d$, the total number of starting points,
    i.e., the isolated points in $\V^*(H^{(d)}(\cdot,1)) = \V^*(F^{(d)}_\square)$ is
    exactly $\mathcal{N}$.
    Our focus is therefore the path between
    $\mathbf{p} = \boldc^*_d$ to $\mathbf{p} = \boldzero$
    parametrized by
    $\mathbf{p}(t) = t \boldc^*_d = (1-t) \boldzero + t \boldc^*_d$.

    For a generic $\boldc^*_d$,
    this path avoids the exceptional set $Q$ in the parameter space
    \cite[Lemma 7.1.2]{SommeseWampler2005Numerical}.
    From \Cref{thm: parameter homotopy}, again,
    as $t$ goes from 1 to 0,
    the nonsingular isolated solutions to $H^{(d)}(\boldx,t) = 0$
    form exactly $\mathcal{N}$ smooth paths (smoothly parametrized by $t$)
    emanating from points in $\V^*_0( F^{(d)}_{\square} )$
    and reach all isolated points of $\V^*_0( F^{(d-1)}_{\square} )$ as limit points.
\end{proof}

The homotopy continuation procedure that tracks the solution paths
defined by $H^{(d)}$ as $t$ moves from 1 to 0
produces both the rank $d-1$ sample superset 
and the starting points for $H^{(d-1)}$.
This chain reaction thus can continue until
$H^{(r+1)}$ produces the rank $r$ (the lowest rank) sample superset for $F$.
This is the stratified polyhedral homotopy.

\begin{definition}[Unmixed stratified polyhedral homotopy]
    \label{def: stratified polyhedral homotopy}
    For an unmixed system $F$ \eqref{equ:unmixed-sys}, in standard form,
    of $q$ Laurent polynomials in $n$ variables $\boldx = (x_1,\dots,x_n)$
    and generic lifting function $\boldomega : \supp(F) \to \Q^+$,
    we define $H = (H_1,\ldots,H_n) : (\C^*)^n \times \C^{q+1} \to \C^n$ given by
    \begin{equation}\label{equ:cascade}
        H_i(\boldx,\boldt) =
        \left\{
        \begin{aligned}
            \left(\boldc_i   + \sum_{k=r+1}^{n} t_{k-r} \lambda_{ik}    \boldc^*_k \right) &\cdot (\boldx^A \circ e^{-M t_0 \boldomega}) 
            \quad\text{for } i = 1,\ldots,q \\
            \left(\boldc_i^* + \sum_{k=r+1}^{n} t_{k-r} \lambda_{q+i,k} \boldc^*_k \right) &\cdot (\boldx^A \circ e^{-M t_0 \boldomega}) 
            \quad\text{for } i = 1,\ldots,r \\
        \end{aligned}
        \right.
    \end{equation}
    where $\boldt = (t_0,t_1,\dots,t_q)$ and $M$ is a sufficiently large positive real number.
\end{definition}

Here, ``$\circ$'' denotes the entry-wise product between two row vectors
of the same length, which is the group operation for $(\C^*)^m$.
The constant $M \in \R^+$ is the same constant used in \eqref{equ: 2-step polyhedral exponential},
which can be computed from the Newton polytope of $H$.

The starting points of the homotopy paths at $\boldt = (1,\ldots,1)$
can be obtained by the same process that bootstraps the polyhedral homotopy
(a brief review of this process is included in \Cref{appendix: bootstrapping} of the appendix for completeness).
Indeed, all $\C^*$-zeros of $H(\boldx,(1,\ldots,1))$ are isolated and nonsingular
and the total number is exactly the \emph{normalized volume}
\[
    n! \evol(\conv(\supp(F))),
\]
of the common Newton polytope $\conv(\supp(F))$.
To obtain sample supersets of all ranks,
we could apply the standard continuation procedure 
along the piecewise-linear parameter path
\[
    (1,\dots,1) \to
    (0,1,\dots,1) \to
    (0,0,1,\dots,1) \to
    \cdots
    (0,\dots,0,1) \to
    (0,\dots,0),
\]
in the $\boldt$-space,
starting from the initial points provided by the
bootstrapping process of polyhedral homotopy.
The parameter path consists of $q+1$ linear segments.
At the end of each segment, the projection of the solution paths
onto the $\boldx$-coordinates generates the sample supersets for $F$
of ranks $n,n-1,\dots,n-q$.
We summarize this algorithm in \Cref{alg:unmixed stratified}.

\begin{algorithm}
    \caption{Unmixed stratified polyhedral homotopy algorithm for regular zeros}
    \label{alg:unmixed stratified}
    \begin{algorithmic}[1]
        \REQUIRE{
            An unmixed system $F$ of $q$ Laurent polynomials in standard form,
            lifting function $\boldomega : S \to \Q^+$ with generic images,
            and generic complex vectors $\boldc_1^*,\ldots,\boldc_n^* \in \C^m$
        }
        \ENSURE{
            Returns finite sample sets $(W_n,\ldots,W_r)$ such that,
            for $d = r,\ldots,n$,
            $W_d \subseteq \V^*_d(F)$ and $W_d$ intersects each $d$-dimensional reduced irreducible component of $\V^*(F)$.
        }
        \STATE{Define $\tilde{X}_{-1} = \text{PolyhedralBootstrap}(F,\boldomega)$}
        \STATE{Define $W_{n+1} = \varnothing$}
        \STATE{Define $\boldt = (t_0,t_1,\dots,t_q) = (1,1,\dots,1)$}
        \FOR{ $k = 0,\dots,q$ }
            \STATE{Define $X_{k} = \text{HomotopyContinuation}(H,\tilde{X}_{k-1} \setminus W_{n-k+1}, \boldt; \, t_k : 1 \to 0)$}
            \label{line: raw end points}
            \STATE{Define $\tilde{X}_{k} = \{ \boldx \in X_k \mid DF_\square^{(n-k)}(\boldx) \text{ is nonsingular} \}$}
            \label{line: regular filter}
            \STATE{Define 
                $W_{n-k} = \{ 
                    \boldx \in \tilde{X}_k \mid F^{(n-k)}(\boldx) = \boldzero
                    \text{ and }
                    \rank DF^{(n-k)}(\boldx) = n - k
                \}$}
            \label{line: sample set}
            \STATE{Let $t_k = 0$}
        \ENDFOR
    \RETURN $(W_n,\dots,W_r)$
    \end{algorithmic}
\end{algorithm}

In this description,
the \textsf{PolyhedralBootstrap} is the subroutine for bootstrapping
the polyhedral homotopy (see \Cref{appendix: bootstrapping})
for $F$ and a generic lifting function $\boldomega$.
It provides the isolated $\C^*$-solutions to
$H(\boldx,(1,\ldots,1)) = \boldzero$.
The subroutine \textsf{HomotopyContinuation}$(H,\tilde{X}_k,\boldt,t_k : 1 \to 0)$
is the standard algorithm for tracking the paths defined by the equation $H = \boldzero$
in $\C^n \times (0,1]$ starting from the points in $\tilde{X}_k$
at $t_k = 1$ toward $t_k \to 0$
while keeping other parameters in $\boldt = (t_0,\ldots,t_q)$ constant.
The limit points within $(\C^*)^n$ are returned.

\begin{remark}
    An important observation is that since the solution paths are continuous,
    under the genericity assumptions,
    once a path reaches an irreducible component of $\V^*(F)$,
    it will never leave.
    Therefore, in \cref{line: raw end points} of \Cref{alg:unmixed stratified},
    it is only necessary to track the paths starting from
    $\tilde{X}_{k-1} \setminus W_{n-k+1}$,
    i.e., the points that are not already in the previous sample sets.
\end{remark}

\subsection{Numerical considerations}\label{sec: numerical}

In practice, homotopy continuation methods are generally implemented numerically,
and the sets $\tilde{X}_{k}$ in \Cref{alg:unmixed stratified}
are only numerical approximations of the zeros in question.
Therefore,
the condition that $F^{(n-k)}(\boldx) = \boldzero$,
in \cref{line: sample set},
and the rank conditions in \cref{line: regular filter,line: sample set}
must be replaced by numerically well-posed conditions.

For example, the condition $F^{(n-k)}(\boldx) = \boldzero$
may be replaced by the numerically meaningful backward error condition
that $F^{(n-k)}_\epsilon(\boldx) = \boldzero$
for some threshold $\epsilon > 0$
and Laurent system $F^{(n-k)}_\epsilon$ with the same support
such that $\| F^{(n-k)}_\epsilon - F^{(n-k)} \| < \epsilon$.

Similarly, the rank condition for the Jacobian matrices
$DF_\square^{(n-k)}(\boldx)$ and $DF^{(n-k)}(\boldx)$
may be replaced by imposing a bound on the condition number of $DF^{(n-k)}(\boldx)$.
A more robust solution is to frame these problems
as well-studied \emph{rank-revealing problems}
\cite{Rank-revealing}.

\subsection{Truly stratified polyhedral homotopy}\label{sec: combined steps}

\Cref{alg:unmixed stratified} is presented to have the steps
operating in serial along the piecewise-linear parameter path.
In practice, this arrangement is neither necessary nor efficient,
since we generally have good a priori knowledge or an educated guess
about the maximum dimension of the zero sets.
At the very least, unless the system $F$ in $n$ variables is trivial,
the dimension of its $\C^*$-zero set must be strictly less than $n$.
In this case, there is no need to directly compute the rank $n$ sample superset,
and it is sufficient to track the solution paths over the modified parameter path
that starts with the line segment
\[
    (1,\ldots,1) \to
    (0,0,1,\ldots,1) \to
    \cdots
\]
in \Cref{alg:unmixed stratified},
i.e., the line segment given by $s \mapsto (s,s,1,\ldots,1)$.
Along this, the polyhedral homotopy and the perturbation of coefficients
operate simultaneously,
and at the end of this line segment,
a rank $n-1$ sample superset is produced.
This should be the default procedure.

In general, if it is known that $\dim \V^*(F)$ is no more than $d_{\max}$,
then it is sufficient to track the solution paths over the parameter path
that starts with the line segment
\[
    (1,\ldots,1) \to
    (\, 
        \underbrace{0,\ldots,0}_{n - d_{\max}+1}
        \,,\,
        \underbrace{1,\ldots,1}_{q - n + d_{\max}}) \to
    \cdots
    \to
    (0,\ldots,0)
\]
At the end of this first segment,
a rank $d_{\max}$ sample superset is produced, which necessarily contains
sample points for each reduced $d_{\max}$-dimensional irreducible component
of $\V^*(F)$.

\section{Reducing general cases to standard unmixed cases}
\label{sec: general cases}

The constructions presented so far require the target Laurent system
to be of the ``standard unmixed form'' (\Cref{def: standard form}).
This section describes how general cases can be handled.
As reviewed in~\Cref{sec: standard unmixed case},
condition~\ref{cond: nonbinomial} eliminates trivial cases,
while
conditions~\ref{cond: constant term} and~\ref{cond: full rank}
can be satisfied by simple transformations.
Therefore, we can focus on satisfying the torsion-free
(condition~\ref{cond: torsion free} in \Cref{def: standard form})
and unmixedness conditions.

\subsection{Lattice reduction for nonstandard unmixed systems}\label{sec: lattice reduction}

To satisfy the torsion-free condition
(Condition \ref{cond: torsion free} in \Cref{def: standard form}),
the invariant factors of the support matrix $A$ must all be 1.
If not, we can construct a new support through a ``saturation'' process
\cite{Sturmfels1996Grobner}:
Let $d_1,\dots,d_n \ne 0$ be these invariant factors,
and let $P \in \zmat{n}{n}$ and $Q \in \zmat{m}{m}$ be the unimodular matrices
in the Smith Normal Form
\begin{align*}
    P A Q &= 
    \begin{bmatrix}
        D & \boldzero
    \end{bmatrix}
    &&\text{where}&
    D &=
    \left[
    \begin{smallmatrix}
        d_1 &        &     \\
            & \ddots &     \\
            &        & d_n
    \end{smallmatrix}
    \right].
\end{align*}
With these, we define matrices
\begin{align}
    L &= P^{-1} D P \in \zmat{n}{n}
    &
    \tilde{A} &= 
    \begin{bmatrix}
        P^{-1} & \boldzero
    \end{bmatrix}
    Q^{-1}
    \in \zmat{n}{m}.
\end{align}
Then $\tilde{A}$ also has full row rank,
and we can verify that
\[
    P\tilde{A}Q = 
    P \begin{bmatrix} P^{-1} & \boldzero \end{bmatrix} Q^{-1} Q
    =
    \begin{bmatrix}
        I & \boldzero
    \end{bmatrix}.
\]
That is, systems with support matrix $\tilde{A}$
would satisfy the torsion-free condition
(Condition~\ref{cond: torsion free} in \Cref{def: standard form}).
We introduce the new variables $\boldy = (y_1,\dots,y_n)$ via the relation
\begin{equation}
    \label{equ:y-x}
    \boldy = \boldx^L
\end{equation}
By \Cref{lem:binomial-roots}, 
this defines a $d$-fold cover over $(\C^*)^n$, where $d=d_1 \cdots d_n = \det L$.
That is, for each $\boldy \in (\C^*)^n$, there are precisely $d$
distinct choices of $\boldx \in (\C^*)^n$ that would satisfy the above equation.
With this change of variables
\[
    \boldy^{\tilde{A}} = (\boldx^L)^{\tilde{A}} = \boldx^{L\tilde{A}} =
    \boldx^{P^{-1}DP[P^{-1} \, \boldzero] Q^{-1}} =
    \boldx^{P^{-1} [ D \, \boldzero] Q^{-1}} =
    \boldx^{A}
\]
Therefore, via the change of variables \eqref{equ:y-x}, we can replace the
original Laurent polynomial system $F$ with support matrix $A$ 
by a new system in $\boldy$ with support matrix $\tilde{A}$
\[
    \tilde{F}(\boldy) =
    \begin{bmatrix}
        \boldc_1 \cdot \boldy^{\tilde{A}} \\
        \vdots \\
        \boldc_q \cdot \boldy^{\tilde{A}}
    \end{bmatrix}
\]
for which the stratified polyhedral homotopy defined in the previous section
can be applied,
and the $\C^*$-zero set $\V^*(F)$ is a $d$-fold cover over $\V^*(\tilde{F})$
defined by the map \eqref{equ:y-x}.

\subsection{Turning mixed cases into unmixed cases}\label{sec: turning mixed}

The description in \Cref{sec: standard unmixed case}
applies only to unmixed Laurent systems,
i.e., systems of Laurent polynomials with a common support.
This constraint can be removed easily by considering
generic linear combinations of the Laurent polynomials:
For a ``mixed'' Laurent system $F = (f_1,\ldots,f_q)$
in which $\supp(f_1),\ldots,\supp(f_q)$ are not identical,
with a generic $q \times q$ matrix $R$,
we can form an equivalent randomized system
\[
    F^R (\boldx) = R \, F(\boldx).
\]
Here, $F(\boldx)$ is considered as a column vector.
These two systems are equivalent in the sense that
$\V^*(F) = \V^*(F^R)$.
Yet, under the genericity assumption,
there is no cancellation of the terms in $R \, F$,
and hence $F^R$ is unmixed.
The stratified polyhedral homotopy construction
described in \Cref{sec: standard unmixed case}
can therefore be applied to the unmixed system $F^R$ instead.

Since the support of $R \, F$ is $S_1 \cup \cdots \cup S_q$,
where $S_i = \supp (f_i)$ for $i=1,\ldots,q$,
the number of paths defined by the stratified polyhedral homotopy,
i.e. the BKK bound of $F^R$, is
\begin{equation}\label{equ: unmixed BKK}
    n! \evol(\conv(S_1 \cup \cdots \cup S_q)).
\end{equation}
In the rest of this paper,
this bound will be referred to as the
\emph{Kushnirenko bound} to emphasize the fact that
the unmixed version of the BKK bound is used.

In summary, the framework developed here can also be applied
to mixed Laurent systems simply by considering random linear combinations
of the Laurent polynomials in the system.
We conclude this section with a few remarks on the more subtle points.

\begin{remark}
    In the case of $q = n$, i.e. $F$ being a square system,
    it is well known that
    \begin{equation}\label{equ: unmixed vs mixed}
        n! \evol(\conv(S_1 \cup \cdots \cup S_n))
        \ge
        \mvol(\conv(S_1),\ldots,\conv(S_n)).
    \end{equation}
    This follows from the monotonicity of the mixed volume function.
    That is, the transformation $F \mapsto RF$
    may or may not increase the BKK bound,
    which is the number of homotopy paths defined by the stratified polyhedral homotopy.
    Conditions for the equality of the two were first discovered by M. Rojas in 1994
    \cite{Rojas1994Convex}.
    Variations of these conditions have since been rediscovered several times
    \cite{BihanSoprunov2019Criteria,Chen2019Unmixing}.
    As listed in Ref.~\cite{Chen2019Unmixing},
    for many important families of Laurent systems
    from applications,
    the two sides of \eqref{equ: unmixed vs mixed} are identical,
    and thus the randomization process
    does not inflate the number of homotopy paths
    one has to track.
\end{remark}

\begin{remark}
    It should be noted that the transformation $F \mapsto RF$
    is not invariant under lattice translations of the supports,
    even though the $\C^*$-zero set they define is:
    For the Laurent system $F = (f_1,\ldots,f_q)$
    and any set of Laurent monomials
    $\boldx^{\boldv_1},\ldots,\boldx^{\boldv_q}$,
    with $\boldv_1,\ldots,\boldv_q \in \Z^n$,
    the Laurent system $(\boldx^{\boldv_1}\,f_1,\ldots,\boldx^{\boldv_q}\,f_q)$
    also has the exact same $\C^*$-zero set.
    Yet, the randomized system
    $R (\boldx^{\boldv_1}\,f_1(\boldx),\ldots,\boldx^{\boldv_q}\,f_q(\boldx))^\top$
    can be quite different from $F^R = R F$.
    In particular, the Kushnirenko bound \eqref{equ: unmixed BKK},
    i.e. the number of paths the stratified polyhedral homotopy will define,
    may be different depending on the choices of $\boldv_1,\ldots,\boldv_q$.
    Finding the optimal choice so that
    $n! \evol(\conv(S_1 + \boldv_1 \cup \cdots \cup S_q + \boldv_q))$
    is minimized is still an open problem.
\end{remark}


\section{Decomposition of the BKK bound}\label{sec: BKK decomposition}

Bernshtein's first theorem (\Cref{thm:bernshtein-a})
states that for a system of $n$ Laurent polynomials
$(f_1,\ldots,f_n)$ in $n$ variables,
the number of isolated zeros in $(\C^*)^n$
is bounded by the mixed volume
$\mvol_n(P_1,\ldots,P_n)$,
where $P_1,\ldots,P_n$ are the Newton polytopes of $f_1,\ldots,f_n$, respectively.
It equals the normalized volume $n!\evol(P)$
in the unmixed case, i.e., when $P_1,\ldots,P_n = P$ (\Cref{thm:kushnirenko}).
This is the BKK bound.
Indeed, for generic coefficients,
all $\C^*$-zeros are isolated,
and this bound is exact.
When positive-dimensional components are present, however,
the number of isolated $\C^*$-zeros will be strictly less than this bound.
A natural question to ask in this situation is
whether it is possible to decompose the BKK bound
as a sum of local contributions from each irreducible component.

This question mirrors the classical question on 
decomposing the B\'ezout number.
In the 17th century,
Newton already observed that
the number of intersections
between two planar curves of degrees $d_1,d_2$
is bounded by $d_1 \cdot d_2$ \cite{Newton1999Principia}.
In the 18th century, B\'ezout proved this upper bound to be exact
when the curves are in general positions,
and the same bound applies to the isolated zeros
of a system of $n$ polynomials in $\mathbb{CP}^n$ \cite{Struik2014Source}.
This is the B\'ezout bound.
Indeed, when there are no positive-dimensional components
and intersections are counted with multiplicities,
this bound is always exact.
When positive-dimensional components are present, however,
the naive interpretation of this bound breaks down.
The search for a decomposition of the B\'ezout bound
into local contributions from all components thus began.

Among a variety of different approaches
in constructing such a decomposition of the B\'ezout bound,
the dynamic approach proposed by F. Severi \cite{Severi1947}
and subsequently corrected by R. Lazarsfeld \cite{Lazarsfeld1981Excess}
is the most relevant here.
By assigning a multiplicity to each
subvariety of the projective zero set of a system,
they established such a decomposition of the B\'ezout bound.

The stratified polyhedral homotopy method
described above produces a similar assignment of
multiplicity as a by-product,
at least for unmixed cases involving reduced components.
First, through a routine application of the Parameter Homotopy Theorem
(\Cref{thm: parameter homotopy}),
we can verify that the number of points in each of the sample sets
$W_n,\ldots,W_0$, produced by \Cref{alg:unmixed stratified},
is independent of the choices of the generic coefficients.

\begin{proposition}
    If all components of $\V^*(F)$ of dimension $d$ are (generically) reduced,
    then for generic choices of $\boldc^*_1,\ldots,\boldc^*_n \in \C^m$,
    the number of distinct points in the rank-$d$ sample set $W_d$ is a constant
    that is independent of the choices of $\boldc^*_1,\ldots,\boldc^*_n$.
\end{proposition}

Since each point in the sample sets is produced by a homotopy path,
and the total number of paths is the Kushnirenko bound
\eqref{equ: unmixed BKK},
by counting points in each $W_{d_i}$,
we have a crude extension of this bound
that counts the contributions from components of each dimension.

\begin{proposition}
    Suppose the $\C^*$-zero set of a Laurent system $(f_1,\ldots,f_n)$
    consists of components $C_{d_1},\ldots,C_{d_\ell} \ne \varnothing$
    where each $C_{d_i}$ is the union of all $d_i$-dimensional components.
    Let $S = \supp(f_1) \cup \cdots \cup \supp(f_n)$ and
    $\jmath(C_{d_i}) = |W_{d_i}|$, then $\jmath(C_{d_i}) > 0$ and
    \begin{equation}\label{equ: crude bkk decomposition}
        \sum_{i=1}^\ell \jmath(C_{d_i})
        \le
        n! \evol_n(\conv(S)).
    \end{equation}
\end{proposition}

This bound can be refined significantly.
By extending the function $\jmath$ to individual irreducible components
in each $C_{d_i}$ via restriction
(see the remark in \Cref{sec: irreducible decomposition} for the connection to
the stronger irreducible decomposition),
we have a more refined decomposition of the Kushnirenko bound
in terms of contributions from irreducible components.

In addition, by broadening the concept of sample points and components
in the above proposition, we can reach an exact decomposition of
the Kushnirenko bound in certain cases.
First, we can take into consideration end points of homotopy paths
that are filtered out by the rank condition in
\cref{line: regular filter} of \Cref{alg:unmixed stratified}
(singular sample points)
as well as divergent paths (sample points at toric infinity), 
and count them with proper multiplicity.
Second, we need to include subvarieties of $\V^*(f_1,\ldots,f_n)$
that may or may not be irreducible components into the left-hand side of
\eqref{equ: crude bkk decomposition},
as long as they attract homotopy paths defined by \Cref{alg:unmixed stratified}.
In other words, we need to include ``distinguished'' subvarieties
as constructed in Ref.~\cite{Fulton1998Intersection}.
The full development of this theoretical aspect is outside the scope
of this paper, which focuses on the numerical aspect of this problem.
We will, instead, illustrate the exact decomposition of the BKK bound
through a few concrete examples in \Cref{sec: examples}
(e.g., equation~\eqref{equ: kuramoto c4 bkk}).

\section{Examples}\label{sec: examples}

In this section, we present results from numerical experiments
in applying the proposed algorithm to compute sample points
of positive-dimensional $\C^*$-zero sets of some well-known Laurent systems.
All experiments are carried out with a proof-of-concept implementation
that uses \textsf{libDH}~\cite{Chen2021GPU} as the path tracker
which utilizes GPU acceleration.
For a system in $n$ unknowns,
we use the stratified polyhedral homotopy with combined steps
as described in \Cref{sec: combined steps},
which ignores the possibility of $n$-dimensional components.

Internally, calculations, with few exceptions that will be noted below,
are carried out strictly in double-precision floating point numbers,
in order to test the robustness of the proposed numerical algorithm.
Therefore, in the following, words such as ``on'', ``in'', and ``reach''
should be interpreted as points or homotopy paths being sufficiently close to
points or positive-dimensional components
up to a tolerance appropriate for double-precision floating point calculations.
Since the goal is to verify the expected behavior against known solution sets,
no certification of the solutions is performed.

\subsection{The running example}

In the running example (Example \ref{ex: running})
we considered
\[
    F(x_1,x_2) =
    \begin{cases}
        (x_1^2 + x_2^2 - 9)(x_1 + x_2 - 3) \\
        (x_1^2 + x_2^2 - 9)(x_1 - x_2 - 1).
    \end{cases}
\]
Its $\C^*$-zero set $\V^*(F)$
consists of the 1-dimensional irreducible component $Q = \V^*(x_1^2 + x_2^2 - 9)$
and the nonsingular point $P = (x_1,x_2) = (2,1) \not\in Q$.
Both are (generically) reduced.

It is easy to verify that the convex hull of the union of the supports
is the simplex defined by $\{ (0,0), (3,0), (0,3) \}$,
which has normalized volume of 9.
That is, its Kushnirenko bound is 9.
Therefore, the stratified polyhedral homotopy defines 9 homotopy paths.

\begin{itemize}
    \item
        At the end of the first stage of the homotopy,
        6 paths reach 6 (nonsingular) rank-1 sample points
        (each reached exactly once) inside the 1-dimensional component $Q$.
    \item 
        The remaining 3 paths continue onto the second stage,
        and one of them reaches one (nonsingular) rank-0 sample point,
        which coincides with the only isolated zero $P = (2,1)$.
        The remaining two paths converge to points in $Q$
        or its projective closure.
\end{itemize}

This shows that by following the homotopy paths defined by a single homotopy,
both sample points of the 1-dimensional component
and the isolated zero can be reached.

\subsection{Algebraic Kuramoto equations on homogeneous networks}

The Kuramoto model \cite{Kuramoto1975Self}
emerged from the study of networks of oscillators,
which can be modeled as collections of points
on the complex plane circling the origin 
while pulling on one another.
Kuramoto proposed a simple yet illuminating dynamical system
governed by
\begin{equation}\label{equ: kuramoto ode}
    \dot{\theta}_i = \omega_i - \sum_{j \sim i} k_{ij} \sin(\theta_i - \theta_j)
    \quad\text{for } i = 0,1,\ldots,N-1.
\end{equation}
Here, $N$ is the number of oscillators,
which are labeled as $i=0,1,\dots,N-1$.
$\theta_i$ is the phase angle of the $i$-th oscillator,
which describes its state,
and $\omega_i$ is its natural frequency (relative to the mean frequency).
$i \sim j$ indicates oscillators $i$ and $j$ are coupled,
in which case the coupling coefficient $k_{ij} = k_{ji}$
quantifies how strongly they influence one another.
Due to the inherent rotational invariance,
we can fix $\theta_0 = 0$,
and discard the equation for $i=0$.

Fundamental to the study of this model
is the problem of finding frequency synchronization configurations,
which are configurations $(\theta_1,\ldots,\theta_{N-1})$ of the network
for which $\dot{\theta}_i = 0$ for all $i = 1,\ldots,N-1$,
i.e., the equilibria of \eqref{equ: kuramoto ode}.
Though the equilibrium equation for \eqref{equ: kuramoto ode}
is not algebraic, with the change of variables
$x_i = e^{\imag \theta_i}$,
the synchronization configurations can be described by the
system of Laurent polynomial equations
\[
    0 = \omega_i - \sum_{i \sim j} \frac{k_{ij}}{2 \imag} \left(
        x_i x_j^{-1} - x_j x_i^{-1}
    \right)
    \quad\text{for } i = 1,\ldots,N-1.
\]
This is the algebraic Kuramoto equation.
Its B\'ezout number and bi-homogeneous B\'ezout number
are $2^{2(N-1)}$ and $\binom{2(N-1)}{N-1}$, respectively
\cite{BaillieulByrnes1982Geometric}.
Its Kushnirenko bound and BKK bound are identical,
and it can be much lower than the B\'ezout numbers for sparse networks.

The network is \emph{homogeneous} if
$\omega_i = 0$ for all $i=0,\ldots,N-1$.
We shall focus on this special case since it was shown in
~\cite{LindbergZachariahBostonLesieutre2022Distribution}
that under the homogeneity assumption,
for specific choices of the coupling coefficients $\{ k_{ij} \}$,
there can be positive-dimensional solution sets.

\subsection{The 4-cycle network}

For a homogeneous network of 4 oscillators that form a 4-cycle,
the corresponding algebraic Kuramoto system is given by
\[
    F_{C_4} =
    \begin{cases}
        - \left( x_1/x_0 - x_0/x_1 \right)
        - \left( x_1/x_2 - x_2/x_1 \right) \\
        - \left( x_2/x_1 - x_1/x_2 \right)
        - \left( x_2/x_3 - x_3/x_2 \right) \\
        - \left( x_3/x_2 - x_2/x_3 \right)
        - \left( x_3/x_0 - x_0/x_3 \right),
    \end{cases}
\]
where $x_0 = 1$ is the constant
that corresponds to the reference phase of the system.
The $\C^*$-zero set $V = \V^*(F_{C_4})$ contains
two nonsingular isolated zeros
$V_0 = \{ (1,1,1), (-1,1,-1) \}$.
There are also three 1-dimensional components
parametrized by the monomial maps
\begin{align*}
        \boldx_1(\lambda) &= \left(-2 \imag \lambda, -1, 2 \imag \lambda\right), &
        \boldx_2(\lambda) &= \left( 2 \imag \lambda, -1, \frac{1}{ 2 \imag \lambda} \right ), &
        \boldx_3(\lambda) &= \left( 1 / 2 \imag \lambda, \frac{-1}{4\lambda^2}, \frac{-1}{2 \imag \lambda} \right),
\end{align*}
respectively.
In addition, there are two embedded points
$E_1 = (-i,-1,i), E_2 = (i,-1,i)$
inside the 1-dimensional components.
Indeed, they are the intersections of $V_{1,1},V_{1,2},V_{1,3}$.
The existence of positive-dimensional components
was discovered by J. Lindberg, A. Zachariah, N. Boston, and B. Lesieutre \cite{LindbergZachariahBostonLesieutre2022Distribution}.
Detailed analysis of the solutions, including their stability properties,
was provided by D. Sclosa \cite{Sclosa2022Kuramoto}.
Here, we utilize this existing knowledge to verify the expected behavior
of the stratified polyhedral homotopy method.

The Kushnirenko bound of this system is 12,
which is identical to its BKK bound
\cite{ChenDavisMehta2018Counting,Chen2019Unmixing}.
Therefore, the stratified polyhedral homotopy defines 12 homotopy paths.

\begin{remark}\label{rmk: bkk vs bezout}
    It is worth noting the significant advantage of the proposed
    stratified polyhedral homotopy method over homotopy methods
    whose complexity is linear in B\'ezout bounds.
    The B\'ezout number of this system is $2^6 = 64$,
    while the bi-homogeneous B\'ezout number is $\binom{6}{3} = 20$.
    The BKK bound is only 12.
    Indeed, as noted in Ref.~\cite{ChenDavisMehta2018Counting},
    the ratio between either B\'ezout number and the BKK bound
    goes to $\infty$ as $N \to \infty$.
\end{remark}

\begin{itemize}
    \item
        At the end of the first stage of the homotopy,
        no rank-2 sample points are produced,
        which signifies that there are no 2-dimensional components
        in the $\C^*$-zero set of this system.
        All 12 paths thus continue to the next stage.

    \item
        At the end of the second stage,
        6 of the 12 paths reach 6 (nonsingular) rank-1 sample points
        inside the 1-dimensional components,
        two sample points on each of the components $V_{1,1},V_{1,2},V_{1,3}$.
        The remaining 6 paths continue to the next stage.

    \item
        At the end of the third stage,
        2 of the 6 remaining paths converge to the two
        nonsingular isolated zeros $(1,1,1)$ and $(-1,1,-1)$, respectively.
        The rest of the paths converge to two of the
        singular points $E_1$ and $E_2$
        (each reached twice).
\end{itemize}
In this case, there are no divergent paths
(i.e., no paths escape $(\C^*)^3$),
and thus, by counting the number of paths reaching each component,
including the two singular points,
we have a full decomposition of the BKK bound
into the local contributions from 7 components:
\begin{equation}\label{equ: kuramoto c4 bkk}
    \mvol(\mathcal{N}_1,\mathcal{N}_2,\mathcal{N}_3) =
    3! \evol_3(\conv(\mathcal{N}_1 \cup \mathcal{N}_2 \cup \mathcal{N}_3)) = 12
    = \underbrace{2 + 2 + 2}_{\substack{\text{1-dimensional}\\\text{components}}}
    + \underbrace{1 + 1    }_{\substack{\text{Isolated}\\\text{points}}}
    + \underbrace{2 + 2    }_{\substack{\text{Singular}\\\text{points}}},
\end{equation}
where $\mathcal{N}_1,\mathcal{N}_2,\mathcal{N}_3$
are the Newton polytopes of three equations, respectively.
This shows that the bound given in \eqref{equ: crude bkk decomposition}
may become an equality when all ``distinguished'' components
are taken into consideration,
thereby providing an exact decomposition of the BKK bound.

\subsection{The 6-cycle network}

Similar to the formulation above,
the algebraic Kuramoto system for the 6-cycle graph
contains 5 equations in 5 complex variables.
It is shown in Ref.~\cite{ChenKorchevskaiaLindberg2022Typical}
that by picking coupling coefficients $k_{ij} = \pm s$
for some $s \in \C^*$ with an odd number of negative choices,
the resulting Laurent system has 10 different 1-dimensional components,
each having a monomial parametrization similar to those given above.
Here, we choose $k_{ij} = 1$ for $\{i,j\} \ne \{0,1\}$
and $k_{01} = k_{10} = -1$.
The corresponding Laurent system is
\[
    F(x_1,x_2,x_3,x_4,x_5) = \left\{
        \begin{aligned}
            +(x_1/x_0 - x_0/x_1) - (x_1/x_2 - x_2/x_1) \\
            -(x_2/x_1 - x_1/x_2) - (x_2/x_3 - x_3/x_2) \\
            -(x_3/x_2 - x_2/x_3) - (x_3/x_4 - x_4/x_3) \\
            -(x_4/x_3 - x_3/x_4) - (x_4/x_5 - x_5/x_4) \\
            -(x_5/x_4 - x_4/x_5) - (x_5/x_0 - x_0/x_5),
        \end{aligned}
    \right.
\]
where $x_0 = 1$ corresponds to the reference phase as before.
The Kushnirenko bound of this system is
$6 \cdot \binom{6-1}{\lfloor (6-1)/2 \rfloor} = 60$.
Therefore, the stratified polyhedral homotopy defines 60 paths
(in contrast with the B\'ezout number of 1024
or the bi-homogeneous B\'ezout number of 252).

\begin{itemize}
    \item
        No rank-$d$ sample points are produced for all $d > 1$.

    \item
        20 paths reach 20 (nonsingular) rank-1 sample points
        on the 1-dimensional components
        with two sample points on each component.

    \item
        The remaining 40 homotopy paths continue on
        and they reach isolated zeros of $F$
        as well as singular points in positive-dimensional components.
\end{itemize}

Together, these results provide numerical verifications of the results
developed in Ref.~\cite{ChenKorchevskaiaLindberg2022Typical}.
Indeed, they provide strong numerical evidence suggesting that
the positive-dimensional components described in
\cite[Proposition 5.2]{ChenKorchevskaiaLindberg2022Typical}
are the only positive-dimensional components.

\subsection{Nested distinguished components}

In Ref.~\cite{BatesEklundHauensteinPeterson2021Excess},
the polynomial system
\[
    F(x,y,z) =
        \begin{cases}
            (xy - z)(x - y)(x + y - z) \\
            (xy - z)(xy - z + (x - y)(x + 2y - 3z)) \\
            (xy - z)(xy - z + (x - y)(2x - 3y + z))
        \end{cases}
\]
is used as an example.
The $\C^*$-zero set of $F$ consists of
a quadratic surface $Q = \V^*(xy - z)$
and the isolated point $P = (2/11 , 10/11, 12/11) \not\in Q$.
There is also a distinguished 1-dimensional component
$C = \V^*(x-y,xy-z)$ that is contained in $Q$.
Let $S_1,S_2,S_3$ be the supports of the three Laurent polynomials in this system,
then the Kushnirenko bound is $3! \evol_3 (\conv(S_1 \cup S_2 \cup S_3)) = 12$.
Therefore, the stratified polyhedral homotopy method
defines 12 paths.

\begin{itemize}
    \item
        At the end of the first stage,
        11 paths converge to points in $Q$.
        However, not all of them produce nonsingular rank-2 sample points.
        Among them, two pairs of paths converge to two points in $C$
        (each reached twice).
    \item 1 path continues on and converges to $P$.
\end{itemize}

The important observation is that the existence of such
a nested distinguished component does not prevent
the stratified polyhedral homotopy from reaching nonsingular
sample points for the 2-dimensional component
and the isolated zero.
Indeed, such a 1-dimensional distinguished component
contained inside a 2-dimensional distinguished component
can still be sampled, if we take into consideration
the singular sample points.

\subsection{Cyclic-4 system}

The ``Cyclic-$n$'' family \cite{Backelin1989Square} of polynomial systems
has been used as standard test cases relating to solving polynomial systems.
Among this family, the ``Cyclic-4'' system is the smallest system
that has a positive-dimensional zero set.
It is given by
\[
    F(x_1,x_2,x_3,x_4) =
        \begin{cases}
            x_1 + x_2 + x_3 + x_4 \\
            x_1 x_2 + x_2x_3 + x_3x_4 + x_4x_1 \\
            x_1 x_2 x_3 + x_2 x_3 x_4 + x_3 x_4 x_1 + x_4 x_1 x_2 \\
            x_1 x_2 x_3 x_4 - 1.
        \end{cases}
\]
Its $\C^*$-zero set consists of two one-dimensional components
as well as 8 embedded points.
The Kushnirenko bound is 22.
Therefore, the stratified polyhedral homotopy defines 22 paths.
\begin{itemize}
    \item
        No (nonsingular) rank-$d$ sample points are produced for $d > 1$.
        This agrees with the fact that there are no components
        in $\V^*(F)$ of dimension greater than 1.

    \item
        At least 4 (nonsingular) rank-1 sample points are produced,
        two on each of the 1-dimensional components.
        In addition, 2 paths reach end points
        that are numerically singular
        (the condition number of $DF_{\square}^{1}$ exceeds $10^6$).

    \item
        No (nonsingular) rank-0 sample point is produced.
        But 16 paths reach the 8 singular embedded points
        of $\V^*(F)$. Each is reached twice.
\end{itemize}

\noindent
This gives a clear demonstration of the strength of the proposed method.
When the classical polyhedral homotopy is applied directly to solve this system,
only the 8 embedded points are reached,
which are singular zeros of $\V^*(F)$.
In contrast, the stratified polyhedral homotopy
produces nonsingular sample points
on each curve in $\V^*(F)$,
which can be used as input for higher level algorithms
(e.g., irreducible decomposition, as noted in \Cref{sec: irreducible decomposition}).

\section{Concluding remarks}\label{sec: conclusion}

The proposed stratified polyhedral homotopy method computes
a special type of sample points for all reduced irreducible components
of the $\C^*$-zero sets of a Laurent polynomial system.
More specifically, when applied to a Laurent polynomial system $F$
in $n$ complex variables,
the proposed homotopy defines a finite number of piecewise smooth
homotopy paths in $(\C^*)^n$ (or a suitable compactification of it)
that pass through finite sample sets
$W_n, W_{n-1},\ldots,W_1,W_0$ (which may be empty)
such that $W_d$ contains at least one point
from each $d$-dimensional reduced irreducible component of
the $\C^*$-zero set of $F$.
Moreover, such sample points are smooth points
in the sense that
the nullity of the Jacobian matrix of $F$ at these sample points matches
the local dimensions of the components there.
This smoothness property is important,
as it enables these sample points to generate
additional information about the $\C^*$-zero set of $F$
through higher level algorithms in numerical algebraic geometry.
We conclude with a few remarks on 
these higher-level algorithms
that can use sample points produced by
the proposed stratified polyhedral homotopy as input.

\subsection{From sample sets to irreducible decomposition}\label{sec: irreducible decomposition}

At each iteration of \cref{line: sample set} of
\Cref{alg:unmixed stratified},
a finite set of points $W_d$ is produced.
Collectively, they form a numerically well-behaved representation
of the $d$-dimensional components $V_d$ of the $\C^*$-zero set $\V^*(F)$.
Therefore, the production of the sample sets $W_n,W_{n-1},\dots,W_1,W_0$
is numerically equivalent to decomposing $\V^*(F)$
according to the dimensions of its components.
A more refined decomposition is the \emph{irreducible decomposition}.
In particular, the $d$-dimensional component $V_d$
may be further decomposed into its irreducible components
\[
    V_d = V_{d,1} \cup V_{d,2} \cup \cdots \cup V_{d,m_d}.
\]
Under the assumption that these components are reduced,
the numerical counterpart to this decomposition will be
a partition of the rank $d$ sample set $W_d$
\[
    W_d = W_{d,1} \cup W_{d,2} \cup \cdots \cup W_{d,m_d}
\]
such that $W_{d,i} \subset V_{d_i}$ for each $i=1,\ldots,m_d$.
In principle, this partition may be produced through
a \emph{monodromy} algorithm \cite{SommeseVerscheldeWampler2001Monodromy}.
The effectiveness and efficiency of such an approach
will be important questions for future studies.

\subsection{Sampling nonreduced components}

This paper has focused only on (generically) reduced components.
In general, the $\C^*$-zero set of a Laurent system $F$
may contain nonreduced components.
That is, over a component $V$ of the zero set,
it is possible for the Jacobian matrix $DF$
to have a nullity that is strictly greater than
the dimension of a component $V$ at every point.
Such nonreduced components may result in
isolated but singular end points in the set $X_k$
in \cref{line: raw end points} of \Cref{alg:unmixed stratified}.
These points are filtered out in \cref{line: regular filter}.
Consequently, the proposed algorithm simply ignores
the existence of nonreduced components.

The main reason for ignoring such nonreduced components
is that singular end points
in \cref{line: raw end points} of \Cref{alg:unmixed stratified}
(i.e., points in $X_k \setminus \tilde{X}_k$)
may become start points of ``singular'' homotopy paths
in the homotopy continuation step in \cref{line: raw end points}
for which \emph{basic} path tracking algorithms cannot be applied.

While it is possible to apply more advanced algorithms
to tracking such ``singular'' homotopy paths
\cite{SommeseVerscheldeWampler2002Method}
and potentially reach singular sample points
that serve as numerical representations of certain nonreduced components,
within the numerical algebraic geometry community, however,
it is much preferred to replace the equations
that define the same zero set
so that the nonreduced structure on the zero set disappears.
These are special forms of regularization processes.
The most commonly used is a family of closely related
symbolic preprocessing steps collectively known as \emph{deflation}
\cite{DaytonZeng2005Computing,LeykinVerscheldeZhao2006Newton}.
Combining the algorithm proposed here
with deflation steps will be a natural extension
that should be investigated.

\section*{Acknowledgement}

The author is grateful to Taylor Brysiewicz
for pointing out Maurice Rojas's theorems on the monotonicity of the mixed volume,
an oversight in the author's previous work.
He also thanks Frank Sottile for clarifying the connection to excess intersections.


\bibliographystyle{siamplain}
\bibliography{library}

\begin{thebibliography}{10}

\bibitem{AdrovicVerschelde2013Polyhedral}
{\sc D.~Adrovic and J.~Verschelde}, {\em Polyhedral methods for space curves exploiting symmetry applied to the cyclic -roots problem}, Lecture Notes in Computer Science (including subseries Lecture Notes in Artificial Intelligence and Lecture Notes in Bioinformatics), 8136 LNCS (2013), pp.~10--29, \url{https://doi.org/10.1007/978-3-319-02297-0_2}, \url{https://link.springer.com/chapter/10.1007/978-3-319-02297-0_2}.

\bibitem{Backelin1989Square}
{\sc J.~Backelin}, {\em Square multiples $n$ give infinitely many cyclic $n$-roots}, Reports~8, Matematiska Institutionen, Stockholms Universitet, 1989.

\bibitem{BaillieulByrnes1982Geometric}
{\sc J.~Baillieul and C.~I. Byrnes}, {\em Geometric critical point analysis of lossless power system models}, IEEE Transactions on Circuits and Systems, 29 (1982), pp.~724--737, \url{https://doi.org/10.1109/TCS.1982.1085093}.

\bibitem{BatesEklundHauensteinPeterson2021Excess}
{\sc D.~J. Bates, D.~Eklund, J.~D. Hauenstein, and C.~Peterson}, {\em {Excess intersections and numerical irreducible decompositions}}, 2021 23rd International Symposium on Symbolic and Numeric Algorithms for Scientific Computing (SYNASC), 00 (2021), pp.~52--60, \url{https://doi.org/10.1109/synasc54541.2021.00021}.

\bibitem{bates_efficient_2011}
{\sc D.~J. Bates, J.~D. Hauenstein, and A.~J. Sommese}, {\em Efficient path tracking methods}, Numerical Algorithms, 58 (2011), pp.~451--459, \url{https://doi.org/10.1007/s11075-011-9463-8}.

\bibitem{BatesHauensteinSommeseWampler2013Numerically}
{\sc D.~J. Bates, J.~D. Hauenstein, A.~J. Sommese, and C.~W. Wampler}, {\em Numerically Solving Polynomial Systems with Bertini}, Society for Industrial and Applied Mathematics, 2013.

\bibitem{BeltramettiSommese1995Adjunction}
{\sc M.~C. Beltrametti and A.~J. Sommese}, {\em The adjunction theory of complex projective varieties}, vol.~16, Walter de Gruyter, 1995.

\bibitem{Bernshtein1975Number}
{\sc D.~N. Bernshtein}, {\em The number of roots of a system of equations}, Functional Analysis and its Applications, 9 (1975), pp.~183--185.

\bibitem{BihanSoprunov2019Criteria}
{\sc F.~Bihan and I.~Soprunov}, {\em Criteria for strict monotonicity of the mixed volume of convex polytopes}, Advances in Geometry, 19 (2019), pp.~527--540, \url{https://doi.org/10.1515/advgeom-2018-0024}.

\bibitem{BreidingTimme2017HomotopyContinuationjl}
{\sc P.~Breiding and S.~Timme}, {\em Homotopycontinuation.jl: A package for homotopy continuation in julia}, International Congress on Mathematical Software,  (2017), \url{https://doi.org/10.1007/978-3-319-96418-8_54}.

\bibitem{Rank-revealing}
{\sc S.~Chandrasekaran and I.~C.~F. Ipsen}, {\em On rank-revealing factorisations}, SIAM Journal on Matrix Analysis and Applications, 15 (1994), pp.~592--622, \url{https://doi.org/10.1137/S0895479891223781}.

\bibitem{Chen2019Unmixing}
{\sc T.~Chen}, {\em Unmixing the mixed volume computation}, Discrete and Computational Geometry,  (2019), \url{https://doi.org/10.1007/s00454-019-00078-x}, \url{http://arxiv.org/abs/1703.01684}.

\bibitem{Chen2021GPU}
{\sc T.~Chen}, {\em {GPU-accelerated path tracker for polyhedral homotopy}}, arXiv,  (2021), \url{https://doi.org/10.48550/arxiv.2111.14317}, \url{https://arxiv.org/abs/2111.14317}.

\bibitem{ChenDavisMehta2018Counting}
{\sc T.~Chen, R.~Davis, and D.~Mehta}, {\em Counting equilibria of the kuramoto model using birationally invariant intersection index}, SIAM Journal on Applied Algebra and Geometry, 2 (2018), pp.~489--507, \url{https://doi.org/10.1137/17M1145665}, \url{https://epubs.siam.org/doi/10.1137/17M1145665}.

\bibitem{ChenKorchevskaiaLindberg2022Typical}
{\sc T.~Chen, E.~Korchevskaia, and J.~Lindberg}, {\em {On the typical and atypical solutions to the Kuramoto equations}}, arXiv,  (2022), \url{https://doi.org/10.48550/arxiv.2210.00784}, \url{https://arxiv.org/abs/2210.00784}.

\bibitem{ChenLeeLi2014Hom4PS-3}
{\sc T.~Chen, T.-L. Lee, and T.-Y. Li}, {\em Hom4ps-3: A parallel numerical solver for systems of polynomial equations based on polyhedral homotopy continuation methods}, 1 2014, \url{http://link.springer.com/chapter/10.1007/978-3-662-44199-2_30}.

\bibitem{ChenLi2015Homotopy}
{\sc T.~Chen and T.-Y. Li}, {\em Homotopy continuation method for solving systems of nonlinear and polynomial equations}, Commun. Inf. Syst., 15 (2015), pp.~119--307, \url{https://doi.org/10.4310/CIS.2015.v15.n2.a1}.

\bibitem{DaytonZeng2005Computing}
{\sc B.~H. Dayton and Z.~Zeng}, {\em {Computing the multiplicity structure in solving polynomial systems}}, Proceedings of the 2005 international symposium on Symbolic and algebraic computation - ISSAC '05,  (2005), pp.~116--123, \url{https://doi.org/10.1145/1073884.1073902}.

\bibitem{Fulton1998Intersection}
{\sc W.~Fulton}, {\em Intersection Theory}, Springer New York, 1 1998, \url{http://link.springer.com/chapter/10.1007/978-1-4612-1700-8_1}.

\bibitem{GelfandKapranovZelevinsky1994Discriminants}
{\sc I.~M. Gelfand, M.~M. Kapranov, and A.~V. Zelevinsky}, {\em Discriminants, Resultants, and Multidimensional Determinants}, Birkhäuser Boston, 1 1994, \url{https://doi.org/10.1007/978-0-8176-4771-1}, \url{http://link.springer.com/10.1007/978-0-8176-4771-1}.

\bibitem{gunji_phompolyhedral_2004}
{\sc T.~Gunji, S.~Kim, M.~Kojima, A.~Takeda, K.~Fujisawa, and T.~Mizutani}, {\em Phom – a polyhedral homotopy continuation method for polynomial systems}, Computing, 73 (2004), pp.~57--77, \url{https://doi.org/http://dx.doi.org/10.1007/s00607-003-0032-4}.

\bibitem{HauensteinSommese2017What}
{\sc J.~D. Hauenstein and A.~J. Sommese}, {\em What is numerical algebraic geometry?}, Journal of Symbolic Computation, 79 (2017), pp.~499--507, \url{https://doi.org/10.1016/J.JSC.2016.07.015}.

\bibitem{HuberSturmfels1995Polyhedral}
{\sc B.~Huber and B.~Sturmfels}, {\em A polyhedral method for solving sparse polynomial systems}, Mathematics of Computation, 64 (1995), pp.~1541--1555, \url{https://doi.org/10.1090/S0025-5718-1995-1297471-4}.

\bibitem{KimKojima2004Numerical}
{\sc S.~Kim and M.~Kojima}, {\em Numerical stability of path tracing in polyhedral homotopy continuation methods}, Computing, 73 (2004), pp.~329--348, \url{https://doi.org/10.1007/s00607-004-0070-6}.

\bibitem{Kuramoto1975Self}
{\sc Y.~Kuramoto}, {\em Self-entrainment of a population of coupled non-linear oscillators}, 1975, \url{http://link.springer.com/chapter/10.1007/BFb0013365}.

\bibitem{Kushnirenko1975Newton}
{\sc A.~G. Kushnirenko}, {\em A newton polyhedron and the number of solutions of a system of k equations in k unknowns}, Usp. Math. Nauk, 30 (1975), pp.~266--267.

\bibitem{Lazarsfeld1981Excess}
{\sc R.~Lazarsfeld}, {\em {Excess intersection of divisors}}, Compositio Mathematica,  (1981).

\bibitem{LeeLiTsai2008HOM4PS-2.0}
{\sc T.~L. Lee, T.~Y. Li, and C.~H. Tsai}, {\em Hom4ps-2.0: A software package for solving polynomial systems by the polyhedral homotopy continuation method}, Computing (Vienna/New York), 83 (2008), pp.~109--133, \url{https://doi.org/10.1007/s00607-008-0015-6}.

\bibitem{LeykinVerscheldeZhao2006Newton}
{\sc A.~Leykin, J.~Verschelde, and A.~Zhao}, {\em {Newton's method with deflation for isolated singularities of polynomial systems}}, Theoretical Computer Science, 359 (2006), pp.~111--122, \url{https://doi.org/10.1016/j.tcs.2006.02.018}.

\bibitem{LeykinVerscheldeZhuang2006Parallel}
{\sc A.~Leykin, J.~Verschelde, and Y.~Zhuang}, {\em Parallel homotopy algorithms to solve polynomial systems}, Lecture Notes in Computer Science, 4151 LNCS (2006), pp.~225--234, \url{https://doi.org/10.1007/11832225_22}.

\bibitem{Li2003Numerical}
{\sc T.-Y. Li}, {\em Numerical solution of polynomial systems by homotopy continuation}, 2003, \url{https://doi.org/10.1016/S1570-8659(02)11004-0}.

\bibitem{LiSauerYorke1987Random}
{\sc T.~Y. Li, T.~Sauer, and J.~A. Yorke}, {\em The random product homotopy and deficient polynomial systems}, Numerische Mathematik, 51 (1987), pp.~481--500, \url{https://doi.org/10.1007/BF01400351}.

\bibitem{LiSauerYorke1989Cheater}
{\sc T.-Y. Li, T.~Sauer, and J.~A. Yorke}, {\em The cheater's homotopy: an efficient procedure for solving systems of polynomial equations}, SIAM Journal on Numerical Analysis,  (1989), pp.~1241--1251.

\bibitem{LiWang1996BKK}
{\sc T.-Y. Li and X.~Wang}, {\em The {BKK} root count in $\mathbb{C}^n$}, Mathematics of Computation, 65 (1996), pp.~1477--1485, \url{https://doi.org/10.1090/S0025-5718-96-00778-8}.

\bibitem{LindbergZachariahBostonLesieutre2022Distribution}
{\sc J.~Lindberg, A.~Zachariah, N.~Boston, and B.~Lesieutre}, {\em {The Distribution of the Number of Real Solutions to the Power Flow Equations}}, IEEE Transactions on Power Systems, PP (2022), pp.~1--1, \url{https://doi.org/10.1109/tpwrs.2022.3170232}.

\bibitem{MorganSommese1989Coefficient}
{\sc A.~P. Morgan and A.~J. Sommese}, {\em Coefficient-parameter polynomial continuation}, Applied Mathematics and Computation, 29 (1989), pp.~123--160, \url{https://doi.org/10.1016/0096-3003(89)90099-4}.

\bibitem{Newton1999Principia}
{\sc I.~Newton}, {\em The Principia: Mathematical Principles of Natural Philosophy}, University of California Press, Berkeley, 1999.

\bibitem{Rojas1999Solving}
{\sc J.~Rojas}, {\em {Solving Degenerate Sparse Polynomial Systems Faster}}, Journal of Symbolic Computation, 28 (1999), pp.~155--186, \url{https://doi.org/10.1006/jsco.1998.0271}.

\bibitem{Rojas1994Convex}
{\sc J.~M. Rojas}, {\em A convex geometric approach to counting the roots of a polynomial system}, Theoretical Computer Science, 133 (1994), pp.~105--140, \url{https://doi.org/10.1016/0304-3975(93)00062-A}.

\bibitem{RojasWang1996Counting}
{\sc M.~J. Rojas and X.~Wang}, {\em Counting affine roots of polynomial systems via pointed newton polytopes}, Journal of Complexity, 12 (1996), pp.~116--133, \url{https://doi.org/10.1006/jcom.1996.0009}.

\bibitem{Sclosa2022Kuramoto}
{\sc D.~Sclosa}, {\em {Kuramoto Networks with Infinitely Many Stable Equilibria}}, arXiv,  (2022), \url{https://arxiv.org/abs/2207.08182}.

\bibitem{Severi1947}
{\sc F.~Severi}, {\em {Il concetto generale di molteplicità delle soluzioni pei sistemi di equazioni algebriche e la teoria dell'eliminazione}}, Annali di Matematica Pura ed Applicata, 26 (1947), pp.~221--270, \url{https://doi.org/10.1007/bf02415380}.

\bibitem{SommeseVerschelde2000Numerical}
{\sc A.~J. Sommese and J.~Verschelde}, {\em Numerical homotopies to compute generic points on positive dimensional algebraic sets}, Journal of Complexity, 16 (2000), pp.~572--602, \url{https://doi.org/10.1006/jcom.2000.0554}.

\bibitem{SommeseVersheldeWampler2001Numerical}
{\sc A.~J. Sommese, J.~Verschelde, and C.~W. Wampler}, {\em Numerical decomposition of the solution sets of polynomial systems into irreducible components}, SIAM Journal on Numerical Analysis, 38 (2001), pp.~2022--2046, \url{https://doi.org/10.1137/S0036142900372549}.

\bibitem{SommeseVerscheldeWampler2001Monodromy}
{\sc A.~J. Sommese, J.~Verschelde, and C.~W. Wampler}, {\em Using monodromy to decompose solution sets of polynomial systems into irreducible components}, in Applications of algebraic geometry to coding theory, physics and computation, Springer, 2001, pp.~297--315.

\bibitem{SommeseVerscheldeWampler2002Method}
{\sc A.~J. Sommese, J.~Verschelde, and C.~W. Wampler}, {\em A method for tracking singular paths with application to the numerical irreducible decomposition}, Algebraic Geometry, a Volume in Memory of Paolo Francia,  (2002), pp.~329--345.

\bibitem{sommese_symmetric_2002}
{\sc A.~J. Sommese, J.~Verschelde, and C.~W. Wampler}, {\em Symmetric functions applied to decomposing solution sets of polynomial systems}, SIAM Journal on Numerical Analysis, 40 (2002), pp.~2026--2046, \url{http://epubs.siam.org/doi/abs/10.1137/S0036142901397101}.

\bibitem{SommeseVerscheldeWampler2005Introduction}
{\sc A.~J. Sommese, J.~Verschelde, and C.~W. Wampler}, {\em Introduction to numerical algebraic geometry}, Solving Polynomial Equations,  (2005), pp.~301--337, \url{https://doi.org/10.1007/3-540-27357-3_8}.

\bibitem{SommeseWampler1996Numerical}
{\sc A.~J. Sommese and C.~W. Wampler}, {\em Numerical algebraic geometry}, AMS, 1996, pp.~749--763.

\bibitem{SommeseWampler2005Numerical}
{\sc A.~J. Sommese and C.~W. Wampler}, {\em The Numerical Solution of Systems of Polynomials Arising in Engineering and Science}, WORLD SCIENTIFIC, 3 2005, \url{https://doi.org/10.1142/9789812567727}.

\bibitem{Struik2014Source}
{\sc D.~J. Struik}, {\em A source book in mathematics, 1200-1800}, Princeton University Press, 2014.

\bibitem{Sturmfels1996Grobner}
{\sc B.~Sturmfels}, {\em Gr{\"o}bner bases and convex polytopes}, vol.~8, American Mathematical Soc., 1996.

\bibitem{Verschelde1999PHCpack}
{\sc J.~Verschelde}, {\em Algorithm 795: Phcpack: A general-purpose solver for polynomial systems by homotopy continuation}, ACM Transactions on Mathematical Software, 25 (1999), pp.~251--276, \url{https://doi.org/10.1145/317275.317286}, \url{http://doi.acm.org/10.1145/317275.317286}.

\bibitem{Verschelde2009Polyhedral}
{\sc J.~Verschelde}, {\em Polyhedral methods in numerical algebraic geometry}, Contemporary Mathematics, 496 (2009), p.~243.

\bibitem{VerscheldeVerlindenCools1994Homotopies}
{\sc J.~Verschelde, P.~Verlinden, and R.~Cools}, {\em Homotopies exploiting newton polytopes for solving sparse polynomial systems}, SIAM Journal on Numerical Analysis, 31 (1994), pp.~915--930, \url{https://doi.org/10.1137/0731049}, \url{http://epubs.siam.org/doi/10.1137/0731049}.

\bibitem{Zhuang2007Thesis}
{\sc Y.~Zhuang}, {\em Parallel implementation of polyhedral homotopy methods}, 2007.

\bibitem{Zulehner1989Solutions}
{\sc W.~Zulehner}, {\em {On the solutions to polynomial systems obtained by homotopy methods}}, Numerische Mathematik, 54 (1989), pp.~303--317, \url{https://doi.org/10.1007/bf01396764}.

\end{thebibliography}

\appendix
\begin{appendices}

\section{Bootstrapping unmixed polyhedral homotopy}\label{appendix: bootstrapping}
For completeness, we briefly outline, without proofs,
the main procedure for computing the starting points
for the homotopy \eqref{equ:cascade} 
in \Cref{def: stratified polyhedral homotopy},
which are the nonsingular isolated zeros of $H(\boldx,(1,\ldots,1))$.
Without loss of generality, it is sufficient to assume
$F$ is an unmixed square system of $n$ Laurent polynomials in $n$ variables.
Under the genericity assumption for $\omega$,
the projection of the lower hull of 
$\hat{S} = \{ (\bolda,\omega(\bolda)) \mid \bolda \in S \} \subset \Q^{n+1}$
forms a triangulation of $S$.
Let 
\[
    T = \{
        \boldalpha \in \Q^n
        \mid
        (\boldalpha,1) \text{ is an inner normal of a facet of } \hat{S}
    \}.
\]
Then for each $\boldalpha \in T$,
the minimum of the linear functional $\inner{ \bullet }{ (\boldalpha,1) }$
is achieved at a subset $\hat{\Delta}(\boldalpha) \subset \hat{S}$
consisting of exactly $n+1$ points.
Let $\Delta(\boldalpha) \subset S$ be its projection.
Since the columns in the support matrix $A$ and the coefficient matrix $C$
(as in \Cref{rmk: square system})
correspond to points in $S$,
we shall use the notations $A_{\Delta(\boldalpha)}$ and $C_{\Delta(\boldalpha)}$
for the submatrices of $A$ and $C$, respectively,
consisting of columns corresponding to points in $\Delta(\boldalpha)$.
With these, we define
\begin{equation}
    F^{(\boldalpha)}(\boldx) =
    C_{\Delta(\boldalpha)} \, (\mathbf{x}^{A_{\Delta(\boldalpha)}})^\top,
\end{equation}
which is a square system of $n$ Laurent polynomials
each of which is a linear combination of the same $n+1$ monomials.
In Ref.~\cite{LeykinVerscheldeZhuang2006Parallel},
A. Leykin, J. Verschelde, and Y. Zhuang named such a system a ``\emph{simplex system}'',
since its Newton polytope is a simplex.
The numerical issues involved in solving such a system are analyzed in the same article,
and more details are included in the Ph.D. thesis of Y. Zhuang~\cite{Zhuang2007Thesis}.
Through a toric transformation induced by the vector $\boldalpha$,
the solutions to such a simplex system
can be used as numerical approximations
for the starting points of the homotopy paths
for \Cref{alg:unmixed stratified}.

\end{appendices}

\end{document}